\documentclass[reqno,a4paper,11pt]{amsart}

\usepackage[utf8]{inputenc}
\usepackage{graphicx}
\usepackage{subfigure}
\usepackage{amsmath, amssymb, amsthm, mathtools, braket,times}
\usepackage{amsaddr}
\usepackage{enumitem}
\usepackage{cite}
\usepackage{graphicx}
\usepackage{array}
\usepackage{tikz}
\usepackage{mathrsfs}
\usetikzlibrary{arrows}
\usetikzlibrary{shapes}
\usepackage{cleveref}
\usepackage[english]{babel}
\usepackage[T1]{fontenc}
\usepackage{fourier}
\usepackage{bbm}
\usepackage{color}
\usepackage{fullpage}

\newcommand{\whp}{whp}%with high probability}
\newcommand{\Whp}{Whp}%Beginning of a sentence
\newcommand{\prob}[1]{\mathbb{P}\left[#1\right]}    %Probability
\newcommand{\probLarge}[1]{\mathbb{P}\big[#1\big]}    %Probability with large parentheses
\newcommand{\condprob}[2]{\mathbb{P}\left[#1 \;\middle|\; #2\right]}
\newcommand{\variance}[1]{\mathbb{V}\left[#1\right]}
\newcommand{\expec}[1]{\mathbb{E}\left[#1\right]}    %Expectation

\def\Erdos{Erd\H{o}s}
\def\Renyi{R\'enyi}
\def\Luczak{\L{}uczak}
\def\ER{\Erdos--\Renyi}
\def\Bollobas{Bollob\'{a}s}

\newtheorem{thm}{Theorem}[section]

\newtheorem{coro}[thm]{Corollary}
\newtheorem{lem}[thm]{Lemma}

\theoremstyle{remark}

\theoremstyle{definition}
\newtheorem{definition}[thm]{Definition}

\newtheorem*{construction*}{Construction}
\newcommand{\proofof}[1]{\subsection{Proof of \Cref{#1}}}

\crefname{thm}{theorem}{theorems}
\crefname{prop}{proposition}{propositions}
\crefname{coro}{corollary}{corollaries}
\crefname{lem}{lemma}{lemmas}
\crefname{definition}{definition}{definitions}
\crefname{question}{question}{questions}
\crefname{problem}{problem}{problems}
\crefname{conjecture}{conjecture}{conjectures}

\newtheoremstyle{claim}%name
{}%Space above
{}%Space below
{\itshape}%Body font
{}%Indent amount
{\bf}%Theorem head font
{.}%Punctuation after theorem head
{.5em}%Space after theorem head
{}%Theorem head spec(can be left empty, meaning ‘normal’)
\theoremstyle{claim}

\crefname{claim}{claim}{claims}

\def\N{\mathbb{N}} %natural numbers
\def\R{\mathbb{R}} %real numbers
\def\ur{\in_R} %element chosen uniformly at random
\newcommand{\rounddown}[1]{\left\lfloor#1\right\rfloor} %floor function

\def\twoconcentration{D}

\def\bin{\mathcal{B}}
\def\ball{B}
\def\nbins{n} %number of bins
\def\nballs{k} %number of balls
\def\location{\mathbf{A}} %location vector
\def\locationBit{A} %coordinate of a location vector
\def\loadvector{\mathbf{\lambda}} %vector of loads
\def\load{\lambda} %load of a bin, i.e. number of balls in the bin
\def\maxload{\lambda^\ast} %maximum load
\DeclareMathOperator{\BB}{BB}
\DeclareMathOperator{\MBB}{M}
\newcommand{\binsandballs}[2]{\BB\left(#1, #2\right)}
\newcommand{\maxbinsandballs}[2]{\MBB\left(#1, #2\right)}

\newcommand{\concentration}[2]{\nu\left(#1, #2\right)} %unique solution of equation k^x*e^x/(x^(x+1/2)n^(x-1))=1
\newcommand{\specialconcentration}[1]{\nu\left(#1\right)} %unique solution of equation above for the special case k=n
\def\Concentration{\nu}

\def\ntrees{t} %number of trees
\def\rootrv{Z} 
 
\def\rootvectorbit{Y}

\newcommand{\smallo}[1]{o\left(#1\right)}
\newcommand{\bigo}[1]{O\left(#1\right)}
\newcommand{\smallomega}[1]{\omega\left(#1\right)}
\newcommand{\Th}[1]{\Theta\left(#1\right)}

\def\planargraph{P} %planar graph
\def\planarclass{\mathcal{P}} %class of planar graphs

\def\multigraph{M} %multigraph

\def\forest{F} %forest of trees with specified roots
\def\forestclass{\mathcal{F}} %class of forests of trees with specified roots

\def\nocomplex{U} %graph without complex component
\def\nocomplexclass{\mathcal{U}} %class of graphs without complex components

\newcommand{\maxdegree}[1]{\Delta\left(#1\right)} %maximum degree
\newcommand{\maxdegreeLarge}[1]{\Delta\big(#1\big)} %maximum degree (with big parentheses)
\newcommand{\largestcomponent}[1]{L_1\left(#1\right)} %largest component
\def\Largestcomponent{L_1} %largest component
\newcommand{\rest}[1]{R\left(#1\right)} %graph without largest component ('rest')
\def\Rest{R} %graph without largest component ('rest')
\def\degreesequence{\mathbf{d}}
\newcommand{\degree}[2]{d_{#2}\left(#1\right)} %degree of a vertex
\newcommand{\vertexSet}[1]{V\left(#1\right)} %vertex set of a graph
\newcommand{\edgeSet}[1]{E\left(#1\right)} %edge set of a graph
\newcommand{\numberVertices}[1]{v\left({#1}\right)} %number of vertices in a graph
\newcommand{\numberEdges}[1]{e\left({#1}\right)} %number of edges in a graph
\newcommand{\numberVerticesLarge}[1]{v\big({#1}\big)} %number of vertices in a graph (with big parentheses)
\newcommand{\numberEdgesLarge}[1]{e\big({#1}\big)} %number of edges in a graph (with big parentheses)

\newcommand{\core}[1]{C\left(#1\right)} %core
\newcommand{\complexpart}[1]{Q\left(#1\right)} %complex part
\newcommand{\restcomplex}[1]{U\left(#1\right)} %graph without complex part
\newcommand{\restcomplexLarge}[1]{U\big(#1\big)}
\def\Restcomplex{U} %graph without complex part
\def\Complexlargestcore{Q_L} %complex part of the largest component of the core
\newcommand{\complexlargestcore}[1]{Q_L\left(#1\right)}
\newcommand{\complexlargestcoreLarge}[1]{Q_L\big(#1\big)}
\def\Complexrestcore{Q_S} %complex part of the rest (=graph without largest component) of the core
\newcommand{\complexrestcore}[1]{Q_S\left(#1\right)}
\newcommand{\complexrestcoreLarge}[1]{Q_S\big(#1\big)}
\def\complexclass{\mathcal{Q}}
\def\complexgraph{Q}

\def\cl{\mathcal{A}} %class of graphs
\def\func{\Phi} %function with domain a set of graphs
\def\seq{\mathbf{a}} %sequence
\def\term{a} %element of the sequence
\newcommand{\condGraph}[2]{#1 \mid #2} %conditional random graph
\def\randomGraph{A}
\def\property{\mathcal{R}}

\newcommand{\contiguous}[2]{#1 \triangleleft #2} %contiguous random variables
\def\prueferseqence{\psi} %Pruefer sequence
\def\prueferinvers{\psi^{-1}} %invers function of the Pruefer sequence
\newcommand{\sequences}[2]{\mathcal{S}\left(#1, #2\right)} %set of Pruefer sequences
\newcommand{\frequency}[2]{\#\left(#1, #2\right)} %number of occurences of an element in a Pruefer sequence

\newcommand{\setbuilder}[2]{\left\{#1 \mid #2\right\}} %set-builder notation

\newcommand{\lessorequal}{\hspace{0.06cm}\leq\hspace{0.06cm}}
\newcommand{\greaterorequal}{\hspace{0.06cm}\geq\hspace{0.06cm}}
\newcommand{\equal}{\hspace{0.06cm}=\hspace{0.06cm}}
\newcommand{\greater}{\hspace{0.06cm}>\hspace{0.06cm}}
\newcommand{\defined}{\hspace{0.06cm}:=\hspace{0.06cm}}

\title{Two point concentration of maximum degree in sparse random planar graphs}
\author{Mihyun Kang, Michael Missethan}
\address{Institute of Discrete Mathematics, Graz University of Technology, Steyrergasse 30, 8010 Graz, Austria}
\email{\{kang,missethan\}@math.tugraz.at}
\thanks{Supported by Austrian Science Fund (FWF): I3747 and W1230}
\begin{document}

\begin{abstract}
Let $\planargraph(n,m)$ be a graph chosen uniformly at random from the class of all planar graphs on vertex set $\left\{1, \ldots, n\right\}$ with $m=m(n)$ edges. We show that in the sparse regime, when $\limsup_{n \to \infty} m/n<1$, with high probability the maximum degree of $\planargraph(n,m)$ takes at most two different values.
\end{abstract}

\maketitle

\section{Introduction and results}\label{sec:intro}

\subsection{Motivation}\label{subsec:motivation}
The \ER\ random graph $G(n,m$), introduced by \Erdos\ and \Renyi\ \cite{erdoes1,erdoes2}, is a graph chosen uniformly at random from the class $\mathcal{G}(n,m)$ of all vertex--labelled simple graphs on vertex set $[n]:=\left\{1, \ldots, n\right\}$ with $m=m(n)$ edges, denoted by $G(n,m)\ur \mathcal{G}(n,m)$. Since its introduction $G(n,m)$, together with the closely related binomial random graph $G(n,p)$, has been intensively studied (see e.g. \cite{rg1, rg2, rg3}). A particularly remarkable feature of this model is the \lq concentration\rq\ of many graph parameters. That is, with high probability (meaning with probability tending to 1 as $n$ tends to infinity, {\em \whp\ } for short) certain graph parameters in $G(n,m)$ lie  in \lq small\rq\ intervals, which only depend on $n$ and $m$.

The graph parameter we will focus on in this paper is the maximum degree of a graph $H$, denoted by $\maxdegree{H}$. \Erdos\ and \Renyi\ \cite{erdoes1} were the first to consider $\maxdegree{G(n,m)}$ and since then, many results on $\maxdegree{G(n,m)}$ and, more generally, the degree sequence of $G(n,m)$ were obtained (see e.g. \cite{max_degree1,max_degree2,max_degree3,degree_sequence1,degree_sequence2,degree_sequence3,vertices_given_degree}). A particular interesting result by \Bollobas\ \cite{vertices_given_degree} is that \whp\ $\maxdegree{G(n,m)}$ is concentrated at two values, provided that $m$ is not too \lq large\rq.

\begin{thm}[{\cite{vertices_given_degree}}]\label{thm:max_degree_ergraph}
Let $m=m(n)=\smallo{n\log n}$ and $G=G(n,m)\ur \mathcal G(n,m)$. Then there exists a $\twoconcentration=\twoconcentration(n)\in \N$ such that \whp\ $\maxdegree{G}\in\left\{\twoconcentration, \twoconcentration+1\right\}$.
\end{thm}

We note that \Bollobas\ \cite{vertices_given_degree} actually considered the binomial random graph $G(n,p)$. But by using standard tools of relating $G(n,m)$ and $G(n,p)$ (see e.g. \cite[Section 1.1]{rg1}) one can translate his result as stated in \Cref{thm:max_degree_ergraph}.

In recent decades various models of random graphs have been introduced by imposing additional constraints to $G(n,m)$, e.g. degree restrictions or topological constraints. In particular, random planar graphs and related structures, like random graphs on surfaces and random planar maps, have attained considerable attention \cite{chap,dks,surface,mcd,chap2,gim,planar,msw,maxdegree_planar1,max_degree_vertex_model,planar_map1,planar_map2,planar_map3,planar1,planar2,planar3,planar4,planar5,planar6,planar7,planar8,planar9,planar10}. McDiarmid and Reed \cite{maxdegree_planar1} considered the so--called $n$-vertex model for random planar graphs, that is, a graph $\planargraph(n)$ chosen uniformly at random from the class of all vertex--labelled simple planar graphs on vertex set $[n]$. They proved that \whp\ $\maxdegree{\planargraph(n)}=\Th{\log n}$. Later Drmota, Giménez, Noy, Panagiotou, and Steger \cite{max_degree_vertex_model} used tools from analytic combinatorics and Boltzmann sampling techniques to show that \whp\ $\maxdegree{\planargraph(n)}$ is concentrated in an interval of length $\bigo{\log\log n}$. In contrast to these results, not much is known about the maximum degree in the random planar graph $\planargraph(n,m)$, which is a graph chosen uniformly at random from the class $\planarclass(n,m)$ of all vertex--labelled simple planar graphs on vertex set $[n]$ with $m=m(n)$ edges. In this paper we show, in the flavour of \Cref{thm:max_degree_ergraph}, that in the sparse regime, when $\limsup_{n \to \infty}m/n<1$, \whp\ $\maxdegree{\planargraph(n,m)}$ is concentrated at two values (see \Cref{thm:main_1sup,thm:main_int,thm:main_sub}). In particular, we have \whp\ $\maxdegree{\planargraph(n,m)}=\left(1+\smallo{1}\right)\log n/\log \log n$ if in addition $\liminf_{n \to \infty}m/n>0$ (see \Cref{cor:maxdegree}).

\subsection{Main results}\label{subsec:main}
In order to state our main results, we need the following definition, where we denote by $\log$ the natural logarithm.
\begin{definition}\label{def:nu}
	Let $\Concentration:\N^2\to \R^+$ be a function such that $\concentration{\nbins}{\nballs}$ is the unique positive zero of
	\begin{align*}
	f(x)=f_{\nbins, \nballs}(x)\defined x\log k+x-(x+1/2)\log x-(x-1)\log n.
	\end{align*}
	In case of $\nbins=\nballs$, we write $\specialconcentration{\nbins}:=\concentration{\nbins}{\nbins}$.
\end{definition}
In \Cref{sub:well_definedness} we will prove that the function $\Concentration$ is well--defined, i.e. $f$ has a unique positive zero. In \Cref{lem:nu} we will provide some important properties of $\Concentration$. In \Cref{sec:balls_bins} we motivate the definition of $\Concentration$ in the context of the balls--into--bins model.

%\begin{definition}\label{def:mu}
%	For $x>0$ we denote by $\sol(x)$ the unique positive solution $y$ of the equation
%	\begin{align*}
%	\log x-y\log y+y-1/2\log y=0.
%	\end{align*}
%\end{definition}

We distinguish three different cases according to which \lq region\rq\ the edge density falls into. The first regime which we consider is when $m\leq n/2+\bigo{n^{2/3}}$.

\begin{thm}\label{thm:main_sub}
Let $\planargraph=\planargraph(n,m)\ur \planarclass(n,m)$, $m=m(n)\leq n/2+\bigo{n^{2/3}}$, and $\varepsilon>0$. Then we have \whp\ $\rounddown{\concentration{n}{2m}-\varepsilon}\lessorequal \maxdegree{\planargraph}\lessorequal \rounddown{\concentration{n}{2m}+\varepsilon}$. In particular, \whp\ $\maxdegree{P}\in \left\{\twoconcentration, \twoconcentration+1\right\}$, where $\twoconcentration=\twoconcentration(n):=\rounddown{\concentration{n}{2m}-1/3}$.
\end{thm}

Next, we consider the case when $m=n/2+s$ for $s=s(n)>0$ such that $s=\smallo{n}$ and $s^3n^{-2}\to\infty$. Kang and \Luczak\ \cite{planar} showed that, in contrast to the case when $m\leq n/2+\bigo{n^{2/3}}$, in this regime \whp\ the largest component of $\planargraph=\planargraph(n,m)$ contains significantly more vertices than the second largest component. Therefore, we provide a concentration result on the maximum degree not only for $\planargraph$, but also for the largest component $\largestcomponent{\planargraph}$ of $\planargraph$ and the \lq rest\rq\ $\rest{\planargraph}:=\planargraph\setminus\largestcomponent{\planargraph}$.

\begin{thm}\label{thm:main_1sup}
Let $\planargraph=\planargraph(n,m)\ur \planarclass(n,m)$, $\Largestcomponent=\largestcomponent{\planargraph}$ be the largest component of $\planargraph$, and $\Rest=\planargraph\setminus\Largestcomponent$. Assume $m=m(n)=n/2+s$ for $s=s(n)>0$ such that $s=\smallo{n}$ and $s^3n^{-2}\to \infty$ and let $\varepsilon>0$. Then \whp
\begin{enumerate}
\item\label{thm:main_1sup1} $\rounddown{\specialconcentration{s}-\varepsilon}+1\lessorequal \maxdegree{\Largestcomponent}\lessorequal \rounddown{\specialconcentration{s}+\varepsilon}+1$;
\item\label{thm:main_1sup2} $\rounddown{\specialconcentration{n}-\varepsilon}\lessorequal \maxdegree{\Rest}\lessorequal \rounddown{\specialconcentration{n}+\varepsilon}$.
\end{enumerate}
In particular, \whp\ $\maxdegree{P}\in \left\{\twoconcentration, \twoconcentration+1\right\}$, where $\twoconcentration=\twoconcentration(n):=\max\left\{\rounddown{\specialconcentration{s}+2/3}, \rounddown{\specialconcentration{n}-1/3}\right\}$.
\end{thm}

Finally, we obtain similar results as in \Cref{thm:main_1sup} but for the case $m=\alpha n/2$, where $\alpha$ tends to a constant in $\left(1,2\right)$.
\begin{thm}\label{thm:main_int}
	Let $\planargraph=\planargraph(n,m)\ur \planarclass(n,m)$, $\Largestcomponent=\largestcomponent{\planargraph}$ be the largest component of $\planargraph$, and $\Rest=\planargraph\setminus\Largestcomponent$. Assume $m=m(n)=\alpha n/2$, where $\alpha=\alpha(n)$ is tending to a constant in $\left(1,2\right)$ and let $\varepsilon>0$. Then \whp
	\begin{enumerate}
		\item $\rounddown{\specialconcentration{n}-\varepsilon}+1\lessorequal \maxdegree{\Largestcomponent}\lessorequal \rounddown{\specialconcentration{n}+\varepsilon}+1$;
		\item $\rounddown{\specialconcentration{n}-\varepsilon}\lessorequal \maxdegree{\Rest}\lessorequal \rounddown{\specialconcentration{n}+\varepsilon}$.
	\end{enumerate}
	In particular, \whp\ $\maxdegree{\planargraph}\in \left\{\twoconcentration, \twoconcentration+1\right\}$, where $\twoconcentration=\twoconcentration(n):=\rounddown{\specialconcentration{n}+2/3}$.
\end{thm}

Combining \Cref{thm:main_1sup,thm:main_int,thm:main_sub} we obtain the following statement on the asymptotic order of $\maxdegree{\planargraph}$.
\begin{coro}\label{cor:maxdegree}
	Let $\planargraph=\planargraph(n,m)\ur \planarclass(n,m)$ and assume $m=m(n)$ such that $\liminf_{n \to \infty} m/n>0$ and $\limsup_{n \to \infty} m/n<1$. Then \whp
	\begin{align*}
	\maxdegree{\planargraph}=\left(1+\smallo{1}\right)\log n/\log \log n.
	\end{align*}
\end{coro}

\subsection{Key techniques}
Our proofs are based on the so--called {\em core--kernel approach} (see e.g. \cite{planar,surface,cycles,evolution_sparse_graphs,cycles_luczak}), which is a decomposition and construction technique for sparse graphs. We construct $\planargraph$ stepwise and analyse each of the steps separately. We start by randomly choosing the so--called core $\core{\planargraph}$ of $\planargraph$, which is a small subgraph of $\planargraph$ and itself a random graph. Next, we use simple graph operations to obtain first the so--called complex part $\complexpart{\planargraph}$ of $\planargraph$ and then $\planargraph$ itself. Then we determine the maximum degree in the core $\core{\planargraph}$ and investigate how this information influences  the maximum degrees in the complex part $\complexpart{\planargraph}$ and in $\planargraph$, respectively. To that end, we relate the construction step from the core $\core{\planargraph}$ to the complex part $\complexpart{\planargraph}$ to the {\em balls--into--bins model} by using {\em random forests with specified roots} and a generalised version of {\em Prüfer sequences}. Similarly, we provide a connection between the step from the complex part $\complexpart{\planargraph}$ to $\planargraph$ and the balls--into--bins model by considering {\em random graphs without complex components}, that are random graphs having at most one cycle in each component, and the {\em \ER\ random graph}. Finally, we obtain concentration of the maximum degrees in $\complexpart{\planargraph}$ and $\planargraph$ by proving that the maximum load of a bin is strongly concentrated.

\subsection{Outline of the paper} 
The rest of the paper is structured as follows. After giving the necessary definitions, notations, and concepts in \Cref{sec:prelim}, we provide our proof strategy in \Cref{sec:strategy}. \Cref{sec:balls_bins} is devoted to the balls--into--bins model, which we use in \Cref{sec:ergraph,sec:forests} to show concentration of the maximum degree in the \ER\ random graph, in a random graph without complex components, and in a random forest with specified roots, respectively. In \Cref{sec:proof} we provide the proofs of our main results. Finally in \Cref{sec:discussion}, we discuss a possible generalisation of our results.

\section{Preliminaries}\label{sec:prelim}
\subsection{Notations for graphs}
We consider only undirected graphs or multigraphs and we always assume that the graphs are vertex--labelled.
\begin{definition}
	Given a (simple or multi) graph $H$ we denote by
	\begin{itemize}
		\item 
		$\vertexSet{H}$ the vertex set of $H$ and
		\item[]
		$\numberVertices{H}$ the order of $H$, i.e. the number of vertices in $H$;
		\item 
		$\edgeSet{H}$ the edge set of $H$ and
		\item[]
		$\numberEdges{H}$ the size of $H$, i.e. the number of edges in $H$;	
		\item 
		$\largestcomponent{H}$ the largest component of $H$;
		\item
		$\rest{H}:=H \setminus \largestcomponent{H}$ the graph obtained from $H$ by deleting the largest component;
		\item
		$\degree{v}{H}$ the degree of a vertex $v\in\vertexSet{H}$. If $\vertexSet{H}=[n]$, then we call $\left(\degree{1}{H}, \ldots, \degree{n}{H}\right)$ the degree sequence of $H$.
	\end{itemize}
\end{definition}
\begin{definition}\label{def:graph_class}
	Given a class  $\cl$ of graphs (e.g. the class of planar graphs), we denote by 
	$\cl(n)$ the subclass of $\cl$ containing the graphs on vertex set $[n]$ 
	and by $\cl(n,m)$ the subclass of $\cl$ containing the graphs on vertex set $[n]$ with $m$ edges, respectively. We write $\randomGraph(n)\ur \cl(n)$ for a graph chosen uniformly at random from $\cl(n)$
	and $\randomGraph(n,m)\ur \cl(n,m)$ for a graph chosen uniformly at random from $\cl(n,m)$, respectively. Throughout the paper, we tacitly assume that $|\cl(n)|$ and $|\cl(n,m)|$ are finite for all considered classes $\cl$ and all $n,m\in \N$.  
\end{definition}

\subsection{Complex part and core}\label{sub:decomposition}
We say that a component of a graph $H$ is {\em complex} if it has at least two cycles. The union of all complex components is called the {\em complex part} $\complexpart{H}$. We call the graph $H$ {\em complex} if all its components are complex. The union of all non--complex components is the {\em non--complex part} $\restcomplex{H}:=H\setminus \complexpart{H}$. The {\em core} $\core{H}$ is the maximal subgraph of $\complexpart{H}$ of minimum degree at least two. We denote by $\complexlargestcore{H}$ the component of $\complexpart{H}$ containing the largest component of the core $\largestcomponent{\core{H}}$. The rest of the complex part is denoted by $\complexrestcore{H}:=\complexpart{H}\setminus\complexlargestcore{H}$. We call $\complexlargestcore{H}$ and $\complexrestcore{H}$ the {\em large complex part} and the {\em small complex part}, respectively. We note that the number of vertices in $\complexlargestcore{H}$ is not necessarily larger than in $\complexrestcore{H}$, but it will be true in most cases we consider. Using this decomposition we can split $H$ into the three disjoint parts $\complexlargestcore{H}$, $\complexrestcore{H}$, and $\restcomplex{H}$, i.e.
\begin{align}\label{eq:23}
H=\complexlargestcore{H}~\dot{\cup}~\complexrestcore{H}~\dot{\cup}~\restcomplex{H}.
\end{align} 
Moreover, we have the relations $\core{\complexlargestcore{H}}=\largestcomponent{\core{H}}$ and $\core{\complexrestcore{H}}=\rest{\core{H}}$.

Later we will construct the large complex part, the small complex part, and the non--complex part of a random planar graph independently of each other. To that end, we will use the following two graph classes.
\begin{definition}\label{def:random_complex_part}
	Let $C$ be a core, i.e. a graph with minimum degree at least two, and $q\in \N$. Then we denote by $\complexclass(C,q)$ the class consisting of complex graphs having core $C$ and vertex set $[q]$. We let $\complexgraph(C,q)\ur\complexclass(C,q)$ be a graph chosen uniformly at random from this class.
\end{definition}

\begin{definition}\label{def:nocomplex}
We denote by $\nocomplexclass$ the class consisting of all graphs without complex components. For $n,m\in\N$ we let $\nocomplexclass(n,m)$ be the subclass of all graphs on vertex set $[n]$ with $m$ edges and we write $\nocomplex(n,m)\ur \nocomplexclass(n,m)$ for a graph chosen uniformly at random from $\nocomplexclass(n,m)$.
\end{definition}

\subsection{Random variables and asymptotic notation}
\begin{definition}
Let $S$ be a finite set and let $Y$ and $Z$ be random variables with values in $S$. Then we say that $Y$ is distributed like $Z$, denoted by $Y\sim Z$, if for all $x \in S$ we have $\prob{Y=x}=\prob{Z=x}$.
\end{definition}
Throughout this paper, we use the standard Landau notation and all asymptotics are taken with respect to $n$, i.e. when $n\to \infty$. In order to express that two random variables have asymptotically a \lq similar\rq\ distribution we use the notion of contiguity.
\begin{definition}
For each $n\in\N$, let $S=S(n)$ be a finite set and let $Y=Y(n)$ and $Z=Z(n)$ be random variables with values in $S$. We say that $Z$ is {\em contiguous} with respect to $Y$, denoted by $\contiguous{Z}{Y}$, if for all sequences $I=I(n)\subseteq S(n)$
\begin{align*}
\left(\lim\limits_{n\to \infty}\prob{Y\in I}=1\right) ~\implies~ \left(\lim\limits_{n\to \infty}\prob{Z\in I}=1\right).
\end{align*}
%We call $Y$ and $Z$ {\em mutually contiguous} (denoted by $\mutuallycontiguous{Z}{Y}$) if $\contiguous{Z}{Y}$ and $\contiguous{Y}{Z}$.
\end{definition}

\subsection{Conditional random graphs}\label{sub:conditional_random_graphs}
Given a class $\cl$ of graphs it sometimes seems quite difficult to directly analyse the random graph $\randomGraph=\randomGraph(n)\ur\cl(n)$. In such cases we will often use the idea of conditional random graphs. Loosely speaking, we split $\cl$ into disjoint subclasses and consider for each subclass $\tilde{\cl}$ the random graph $\tilde{\randomGraph}=\tilde{\randomGraph}(n)\ur \tilde{\cl}(n)$, in other words, the random graph $\randomGraph$ conditioned on the event that $\randomGraph\in \tilde{\cl}$. If we can show that some graph property holds in all these \lq conditional\rq\ random graphs \whp, then \whp\ this property holds also in $\randomGraph$. The following definition and lemma makes that idea more precise.
\begin{definition}\label{def:feasible}
	Given a class $\cl$ of graphs, a set $S$, and a function $\func:\cl\to S$, we call a sequence $\seq=(\term_n)_{n\in \N}$ {\em feasible} for $\left(\cl, \func\right)$ if for each $n \in \N$  there exists a graph $H \in \cl(n)$ such that $\func(H)=\term_n$. Moreover, for each $n\in \N$  we denote by $\left(\condGraph{\randomGraph}{\seq}\right)(n)$  a graph chosen uniformly at random from the set $\left\{H \in \cl(n): \func(H)=\term_n\right\}$.  We will often omit the dependence on $n$ and write just $\condGraph{\randomGraph}{\seq}$ (i.e. \lq $\randomGraph$ conditioned on $\seq$\rq)  instead of $\left(\condGraph{\randomGraph}{\seq}\right)(n)$.
\end{definition}
\begin{lem}[{\cite[Lemma 3.2]{cycles}}]\label{lem:conditional_random_graphs}
	Let $\cl$ be a class of graphs, $S$ a set,  $\func:\cl\to S$ a function, and $\property$ a graph property\footnote{Formally a graph property is a set of graphs.}. Let $\randomGraph=\randomGraph(n)\ur \cl(n)$. If for every sequence $\seq=(\term_n)_{n\in \N}$ that is feasible for $\left(\cl, \func\right)$  we have \whp\ $\condGraph{\randomGraph}{\seq} \in \property$, then we have \whp\ $\randomGraph \in \property$.
\end{lem}

\subsection{Internal structure of a random planar graph}\label{sub:internal_structure}
In the proofs of our main results we will use some results from \cite{surface} on the internal structure of a random planar graph $\planargraph(n,m)$, e.g. maximum degree of the core or the order of the core and the complex part, which are reformulated to simplify asymptotic notation. 
\begin{thm}[{\cite[Theorems 5.1 and 5.4]{surface}}]\label{thm:internal_structure}
	Let $\planargraph=\planargraph(n,m)\ur \planarclass(n,m)$, $C=\core{\planargraph}$ be the core, $\Complexlargestcore=\complexlargestcore{\planargraph}$ the large complex part, $\Complexrestcore=\complexrestcore{\planargraph}$ the small complex part, $\Restcomplex=\restcomplex{\planargraph}$ the non--complex part, and $\Largestcomponent=\largestcomponent{\planargraph}$ the largest component of $\planargraph$. In addition, let $h=h(n)=\smallomega{1}$ be a function tending to $\infty$ arbitrarily slowly. We assume that either $m=n/2+s$ for $s=s(n)>0$ such that $s=\smallo{n}$ and $s^3n^{-2}\to \infty$ or $m=\alpha n/2$, where $\alpha=\alpha(n)$ tends to a constant in $(1,2)$. Then \whp\ $\maxdegree{C}$, $\numberVertices{\largestcomponent{C}}$, $\numberVertices{\Complexlargestcore}$, and $\numberVertices{\Complexrestcore}$ lie in the following ranges.	
\begin{center}
	\def\arraystretch{1.25}
\begin{tabular}{l|c c c c}
& $\maxdegree{C}$ & $\numberVertices{\largestcomponent{C}}$ & $\numberVertices{\Complexlargestcore}$ & $\numberVertices{\Complexrestcore}$
\\
\hline
$m=n/2+s$ for $s>0$, $s=\smallo{n}$, and $s^3n^{-2}\to \infty$ & $3$ &  $\Th{sn^{-1/3}}$ & $\left(2+\smallo{1}\right)s$ & $\bigo{hn^{2/3}}$\\
$m=\alpha n/2$ for $\alpha \to c\in (1,2)$ & $\bigo{h}$ & $\Th{n^{2/3}}$ & $\left(\alpha-1+\smallo{1}\right)n$ & $\bigo{hn^{2/3}}$
\end{tabular}
\end{center}
Moreover, \whp\ 
$\numberEdges{\Restcomplex}=\numberVertices{\Restcomplex}/2+\bigo{h\numberVertices{\Restcomplex}^{2/3}}$ and $\Complexlargestcore=\Largestcomponent$.	
\end{thm}

%\subsection{Properties of $\sol(x)$}
%\begin{lem}\label{lem:mu}
%	Let the function $\sol(x)$ be defined as in \Cref{def:mu}. Then for $x\to\infty$
%	\begin{enumerate}
%		\item\label{lem:mu1} $\sol(x)=\left(1+\smallo{1}\right)\log x/\log\log x$;
%		\item\label{lem:mu2} $\sol(cx)=\sol(x)+\smallo{1}$ for each $c>0$.
%	\end{enumerate}
%\end{lem}

\subsection{Properties of $\concentration{\nbins}{\nballs}$}
We will use the following basic properties of $\concentration{\nbins}{\nballs}$ defined in \Cref{def:nu}. For completeness, we provide a proof of the following statement in \Cref{sub:proof_nu}.
\begin{lem}\label{lem:nu}
Let the function $\concentration{\nbins}{\nballs}$ be defined as in \Cref{def:nu} and $\specialconcentration{\nbins}=\concentration{n}{n}$. Then we have
\begin{enumerate}
\item \label{lem:nu5} $\concentration{\nbins}{\nballs}>1$ for all $\nbins, \nballs \in \N$;
\item \label{lem:nu7} if $\nballs=\nballs(\nbins)=\bigo{n^{1/3}}$, then $\concentration{\nbins}{\nballs}\leq 5/3+\smallo{1}$;
\item \label{lem:nu1} if $\nballs=\nballs(\nbins)=\Th{\nbins}$, then $\concentration{\nbins}{\nballs}=\left(1+\smallo{1}\right)\log \nbins/\log\log\nbins$;
\item \label{lem:nu8} if $\nballs=\nballs(\nbins)=\bigo{\nbins}$, then $\concentration{\nbins}{\nballs}=\smallo{\log n}$;
\item \label{lem:nu6}
if $\nballs=\nballs(\nbins)=\bigo{\nbins}$, then $\concentration{\nbins}{\nballs}=\smallomega{\nballs/\nbins}$;
\item \label{lem:nu4} $\concentration{\nbins}{\nballs}$ is strictly increasing in the argument $\nballs$;
\item \label{lem:nu3} if $\nballs=\nballs(\nbins)=\Th{\nbins}$ and $d=d(n)=\smallo{\nbins\left(\log \log \nbins\right)^2/\log \nbins}$, then $\concentration{\nbins}{\nballs+d}-\concentration{\nbins}{\nballs}=\smallo{1}$;
\item \label{lem:nu9} $\specialconcentration{\nbins}$ is strictly increasing;
\item \label{lem:nu2} if $c=c(\nbins)=\Th{1}$, then $\specialconcentration{c\nbins}=\specialconcentration{\nbins}+\smallo{1}$.
\end{enumerate}
\end{lem}

\section{Proof strategy}\label{sec:strategy}
The general proof idea is to relate the random planar graph $\planargraph=\planargraph(n,m)$ to other models of random graphs, e.g. \ER\ random graph or random forests with specified roots. Then we link these random graphs to the balls--into--bins model (see \Cref{sec:balls_bins} for a formal definition of the balls--into--bins model). In particular, we aim to find relations between the maximum degrees of these random graphs and the maximum load of a bin. Then we will show that \whp\ the maximum load of a bin is concentrated at two values (see \Cref{thm:concentration_balls_bins}). Using that we will deduce that the maximum degrees of the above mentioned random graphs, in particular those of the random planar graph $\planargraph$, are strongly concentrated. In the following we make that idea more precise.

In order to proof \Cref{thm:main_sub}, we use the known fact that with positive probability the \ER\ random graph $G(n,m)$ is planar if $m\leq n/2+\bigo{n^{2/3}}$ (see \Cref{thm:non_complex}). Thus, it suffices to determine $\maxdegree{G(n,m)}$ instead of $\maxdegree{P(n,m)}$, which is done through the following relation between $G(n,m)$ and the balls--into--bins model.

\subsection{\ER\ random graph and the balls--into--bins model}\label{sub:strategy_without_complex}
We note that due to \Cref{thm:max_degree_ergraph} we already know that \whp\ $\maxdegree{G(n,m)}$ is concentrated at two values. \Bollobas\ \cite{vertices_given_degree} provided an implicit formula for these two values. However, it seems difficult to relate these values to concentration results of other random graphs considered in this paper (e.g. \Cref{thm:random_complex_part}\ref{thm:random_complex_part2}). Therefore, we give an alternative proof of \Cref{thm:max_degree_ergraph} by providing a relation to the balls--into--bins model. Given $n$ bins $\bin_1, \ldots, \bin_\nbins$ and $2m$ balls $\ball_1, \ldots, \ball_{2m}$ we denote by $\locationBit_i$ the index of the bin to which the $i$--th ball $\ball_i$ is assigned for each $i\in[2m]$. Then we consider the random multigraph $\multigraph$ with $\vertexSet{\multigraph}=[n]$ and $\edgeSet{\multigraph}=\setbuilder{\left\{\locationBit_{2i-1}, \locationBit_{2i}\right\}}{i \in [m]}$. We will show that with positive probability $\multigraph$ is simple and that conditioned on $\multigraph$ being simple, $\multigraph$ is distributed like $G(n,m)$. Then the concentration of $\maxdegree{G(n,m)}$ follows by the concentration of the maximum load of a bin (see \Cref{thm:G_n_m_bins_balls}).

\subsection{Decomposition and conditional random graphs}\label{sub:strategy_decomposition}
In order to proof \Cref{thm:main_1sup,thm:main_int} we will apply the decomposition technique described in \Cref{sub:decomposition}. We recall that we can split the random planar graph $\planargraph=\planargraph(n,m)$ into the large complex part $\Complexlargestcore=\complexlargestcore{\planargraph}$, the small complex part $\Complexrestcore=\complexrestcore{\planargraph}$ and the non--complex part $\Restcomplex=\restcomplex{\planargraph}$ (see (\ref{eq:23})). We note that \whp\ $\Complexlargestcore$ coincide with the largest component $\largestcomponent{\planargraph}$ (see \Cref{thm:internal_structure}). Thus, it suffices to determine $\maxdegree{\Complexlargestcore}$, $\maxdegree{\Complexrestcore}$, and $\maxdegree{\Restcomplex}$ and use that \whp\ $\maxdegree{\largestcomponent{\planargraph}}=\maxdegree{\Complexlargestcore}$ and  $\maxdegree{\rest{\planargraph}}=\max\left\{\maxdegree{\Complexrestcore}, \maxdegree{\Restcomplex}\right\}$. Instead of doing that directly in $\planargraph$, we will consider another random graph $\tilde{\planargraph}$ and then provide a relation between $\planargraph$ and $\tilde{\planargraph}$.

In order to construct $\tilde{\planargraph}$, we assume that $\tilde{l}, \tilde{r}\in \N$ and a core $\tilde{C}$ with largest component $\largestcomponent{\tilde{C}}$ and rest $\rest{\tilde{C}}$ are given. We choose the large complex part $\complexlargestcore{\tilde{\planargraph}}$, the small complex part $\complexrestcore{\tilde{\planargraph}}$, and the non--complex part $\restcomplex{\tilde{\planargraph}}$, independent of each other. In addition, we want that the core $\core{\tilde{\planargraph}}$ is equal to $\tilde{C}$, $\numberVertices{\complexlargestcore{\tilde{\planargraph}}}=\tilde{l}$, and $\numberVertices{\complexrestcore{\tilde{\planargraph}}}=\tilde{r}$. This motivates the following choices: 
\begin{align}
\complexlargestcore{\tilde{\planargraph}}&=\complexgraph\left(\largestcomponent{\tilde{C}},\tilde{l}\right),\label{eq:25}\\
\complexrestcore{\tilde{\planargraph}}&=\complexgraph\left(\rest{\tilde{C}},\tilde{r}\right),\label{eq:26}\\
\restcomplex{\tilde{\planargraph}}&=\nocomplex(\tilde{u},\tilde{w}),\label{eq:27}
\end{align}
where the random graphs on the right hand side are as defined in \Cref{def:random_complex_part,def:nocomplex}. Furthermore, we define $\tilde{u}:=n-\tilde{l}-\tilde{r}$ and $\tilde{w}:=m-\numberEdges{\tilde{C}}+\numberVertices{\tilde{C}}-\tilde{l}-\tilde{r}$, so that $\tilde{\planargraph}$ has $n$ vertices and $m$ edges.

To relate $\tilde{\planargraph}$ to the original random planar graph $\planargraph$ we use so--called conditional random graphs (see \Cref{sub:conditional_random_graphs}). Roughly speaking, the idea of this method is that if for all \lq typical\rq\ choices of $\tilde{C}$, $\tilde{l}$, and $\tilde{r}$ \whp\ a graph property holds in $\tilde{\planargraph}$, then \whp\ this property holds in $\planargraph$. In order to determine what \lq typical\rq\ choices of $\tilde{C}$, $\tilde{l}$, and $\tilde{r}$ are, we use known results on the internal structure of $\planargraph$ (see \Cref{thm:internal_structure}). For example, if we know that \whp\ the core $\core{\planargraph}$ satisfies a certain structure, e.g. the maximum degree is three or the number of vertices lies in a certain interval, then typical choices of $\tilde{C}$ are those cores having this structure.

Using this relation between $\planargraph$ and $\tilde{\planargraph}$ it suffices to show that the maximum degrees in $\complexlargestcore{\tilde{\planargraph}}$, $\complexrestcore{\tilde{\planargraph}}$, and $\restcomplex{\tilde{\planargraph}}$ are strongly concentrated. By the definition of these random graphs (see (\ref{eq:25})--(\ref{eq:27})) it remains to determine the maximum degrees in $\complexgraph\left(C,q\right)$ and $\nocomplex(n,m)$ for fixed values of $C$, $q$, $n$, and $m$. We will see that if we consider $\nocomplex(n,m)$, then we always have $m=n/2+\bigo{n^{2/3}}$. It is well--known that in this regime the \ER\ random graph $G(n,m)\ur \mathcal G(n,m)$ has with positive probability no complex component (see \Cref{thm:non_complex}). Hence, we can deduce $\maxdegree{\nocomplex(n,m)}$ from results on $\maxdegree{G(n,m)}$ (see \Cref{sub:strategy_without_complex}). In order to find $\maxdegree{\complexgraph\left(C,q\right)}$ we will connect $\complexgraph\left(C,q\right)$ to the balls--into--bins model and use a concentration result of the maximum load of a bin. We will sketch that idea in \Cref{sub:strategy_complex_part}.

\subsection{Random complex part and forests with specified roots}\label{sub:strategy_complex_part}
Let $C$ be a core (on vertex set $[\numberVertices{C}]$) and $q\in \N$. In \Cref{def:random_complex_part} we denoted by $\complexgraph\left(C,q\right)$ a graph chosen uniformly at random from the family of all complex graphs with core $C$ and vertex set $[q]$. Moreover, we let $\forestclass(n, \ntrees)$ be the class of forests on vertex set $[n]$ consisting of $\ntrees$ trees such that each vertex from $[\ntrees]$ lies in a different tree. The elements in $\forestclass(n, \ntrees)$ are called {\em forests with specified roots} and the vertices in $[\ntrees]$ {\em roots}. For simplicity we will often just write forests instead of forest with specified roots. We can construct $\complexgraph=\complexgraph\left(C,q\right)$ by choosing a random forest $\forest=\forest(q, \numberVertices{C})\ur\forestclass(q, \numberVertices{C})$ and replacing each vertex $v$ in $C$ by the tree with root $v$. For the degrees of vertices in $\complexgraph$ we obtain $\degree{v}{\complexgraph}=\degree{v}{C}+\degree{v}{\forest}$ for $v\in C$ and $\degree{v}{\complexgraph}=\degree{v}{\forest}$ otherwise. In our applications we will have that $\maxdegree{C}$ is bounded and $\numberVertices{C}$ is \lq small\rq\ compared to $q$ (see \Cref{thm:internal_structure}). This will imply that \whp\ $\maxdegree{\complexgraph}=\maxdegree{\forest}$ (see \Cref{thm:random_complex_part}). 

In order to determine $\maxdegree{\forest}$ we will introduce a bijection between $\forestclass(n, \ntrees)$ and $\sequences{n}{\ntrees}:=[n]^{n-\ntrees-1}\times [\ntrees]$ similar to Prüfer sequences for trees (see \Cref{sub:pruefer}). Given a forest $\forest\in \forestclass(n,\ntrees)$ we recursively delete the leaf\footnote{A leaf in a forest is a vertex of degree one.} with largest label and thereby build a sequence by noting the unique neighbours of the leaves. We will show in \Cref{thm:pruefer} that this is indeed a bijection and that the degree of a vertex $v$ is determined by the number of occurrences of $v$ in the sequence (see (\ref{eq:12})). It is straightforward to construct a random element from $\sequences{n}{\ntrees}$ by a balls--into--bins model such that the load of a bin equals the number of occurrences in the sequence of the corresponding element. Thus, the concentration result on the maximum load translates to a concentration of the maximum degree $\maxdegree{\forest}$.

\section{Balls into bins}\label{sec:balls_bins}
Balls--into--bins models have been extensively studied in the literature (see e.g. \cite{johnson_kotz, mitzenmacher_upfal}). Throughout the paper, we will use the following model. Given $\nbins$ bins $\bin_1, \ldots, \bin_\nbins$ we sequentially assign $\nballs$ balls $\ball_1, \ldots, \ball_\nballs$ to those $\nbins$ bins by choosing a bin for each ball, independently and uniformly at random. Let $\location=\left(\locationBit_1, \ldots, \locationBit_\nballs\right)$ be the location vector, i.e. $\locationBit_i$ is the index of the bin to which the $i-$th ball $\ball_i$ is assigned. For each $j\in [\nbins]$ we call the number of balls in the $j$--th bin $\bin_j$ the {\em load} $\load_j=\load_j(\location)$. We write $\loadvector=\loadvector(\location)=\left(\load_1, \ldots, \load_\nbins\right)$ for the vector of all loads and denote by $\maxload=\maxload(\location)=\max_{j \in [\nbins]}\load_j$ the maximum load in a bin. For $\ntrees\in[\nbins]$ we let $\maxload_{ \ntrees}=\maxload_{\ntrees}(\location)=\max_{j \in [\ntrees]}\load_j$ be the maximum load in one of the first $\ntrees$ bins $\bin_1, \ldots, \bin_\ntrees$. We write $\binsandballs{\nbins}{\nballs}$ for a random vector distributed like the location vector $\location$ and $\maxbinsandballs{\nbins}{\nballs}$ for a random variable distributed like the maximum load $\maxload$. In order to express that $\location$ is the location vector of a balls--into--bins experiment with $\nbins$ bins and $\nballs$ balls, we often write $\location\sim \binsandballs{\nbins}{\nballs}$.

Gonnet \cite{gonnet} proved in the case $\nbins=\nballs$ that \whp\ $\maxbinsandballs{\nbins}{\nbins}=\left(1+\smallo{1}\right)\log n/\log\log n$. Later Raab and Steger \cite{balls_bins1} considered $\maxbinsandballs{\nbins}{\nballs}$ for different ranges of $\nballs$. Amongst other results, they showed that \whp\ $\maxbinsandballs{\nbins}{\nballs}=\left(1+\smallo{1}\right)\log n/\log\log n$ is still true, as long as $\nballs=\Th{\nbins}$. In the following we improve their result.  More precisely, we show that if $\nballs=\smallo{\nbins\log\nbins}$, then \whp\ $\maxbinsandballs{\nbins}{\nballs}$ is actually concentrated at two values.

Before proving that rigorously, we motivate this result by providing the following heuristic. For $l=l(n)\in\N$ we let $X^{(l)}$ be the number of bins with load $l$. We have 
\begin{align}\label{eq:28}
	\expec{X^{(l)}}=\nbins\binom{\nballs}{l}\left(1/n\right)^l\left(1-1/n\right)^{\nballs-l}=:\mu(l).
\end{align}
We expect that the load $l$ of a bin is much smaller than $\nballs$ and therefore we have
\begin{align*}
	\mu(l)=\Th{1}k^le^ll^{-l-1/2}n^{-l+1}.
\end{align*}
Intuitively, the maximum load $\maxload$ should be close to the largest $l$ for which $\mu(l)=\Th{1}$ is satisfied, in other words, $\log \left(\mu(\maxload)\right)$ should be close to 0. This motivates the definition of $\concentration{\nbins}{\nballs}$ in \Cref{def:nu} as the unique positive zero of the function
\begin{align*}
	f(l)=f_{\nbins, \nballs}(l):=l\log k+l-(l+1/2)\log l-(l-1)\log n,
\end{align*}
which is asymptotically equal to $\log\left(\mu(l)\right)$ up to an additive constant. We will use the first and second moment method (see e.g. \cite{probmethod,rg1}) to make that heuristic rigorous and show that the maximum load $\maxload$ is strongly concentrated around $\concentration{\nbins}{\nballs}$.

%\begin{lem}\label{lem:properties_nu}
%Let $\nballs=\nballs(\nbins)=\Th{\nbins}$, then $\concentration{\nbins}{\nballs}=\left(1+\smallo{1}\right)\log \nbins/\log\log\nbins$.
%\end{lem}

\begin{thm}\label{thm:concentration_balls_bins}
If $\nballs=\nballs(\nbins)=\bigo{\nbins}$ and $\varepsilon>0$, then \whp
\begin{align*}
\rounddown{\concentration{\nbins}{\nballs}-\varepsilon}\lessorequal \maxbinsandballs{\nbins}{\nballs}\lessorequal \rounddown{\concentration{\nbins}{\nballs}+\varepsilon}.
\end{align*}
\end{thm}
\begin{proof}
Let $\location\sim \binsandballs{\nbins}{\nballs}$ be the location vector, $\load_j=\load_j(\location)$ the load of bin $\bin_j$ for each $j\in[\nbins]$, and $\maxload=\maxload(\location)$ the maximum load. First we consider the case $\nballs\leq \nbins^{1/3}$. Then we have
\begin{align}\label{eq:18}
	\prob{\maxload=1}=\prod_{i=1}^{\nballs-1}\left(1-\frac{i}{\nbins}\right)\greaterorequal \left(1-\frac{\nballs}{\nbins}\right)^\nballs=1-\smallo{1}.
\end{align}
Due to \Cref{lem:nu}\ref{lem:nu5} and \ref{lem:nu7} we have $1<\concentration{\nbins}{\nballs}\leq 7/4$ for $\nbins$ large enough. Together with (\ref{eq:18}) this shows the statement for the case $\nballs\leq \nbins^{1/3}$. Hence, it remains to consider the case $\nballs>\nbins^{1/3}$. For $l\in[\nballs]$ and $j\in[\nbins]$ we let $X_j^{(l)}=1$ if $\load_j=l$, i.e. the number $\load_j$ of balls (among $\nballs$ balls) in the $j$--th bin $\bin_j$ is equal to $l$, and $X_j^{(l)}=0$ otherwise. In addition, we let $X^{(l)}=\sum_{j=1}^{\nbins}X_j^{(l)}$ be the number of bins with load $l$. Then we have 
$\prob{X_j^{(l)}=1}=\binom{\nballs}{l}\left(1/n\right)^l\left(1-1/n\right)^{\nballs-l}$ and obtain (\ref{eq:28}). If $l=\bigo{\nballs^{1/2}}$, then
$\binom{\nballs}{l}=\Th{1}\nballs^le^l/l^{l+1/2}$,
where we used Stirling's formula for $l!$. Hence, we get 
\begin{align}\label{eq:13}
	\mu(l)=\Th{1}\frac{\nballs^l e^l}{l^{l+1/2}n^{l-1}},
\end{align}
because $\left(1-1/n\right)^{\nballs-l}=\Theta(1)$. For an upper bound of the maximum load $\maxload$ we will use the first moment method. Let $l^\ast=l^\ast(\nbins):=\rounddown{\concentration{\nbins}{\nballs}+\varepsilon}+1$ and $\delta=\delta(n):=l^\ast-\concentration{\nbins}{\nballs}\geq\varepsilon$. Due to \Cref{lem:nu}\ref{lem:nu8} and the assumption $\nballs> \nbins^{1/3}$ we have $l^\ast=\bigo{\nballs^{1/2}}$. Thus, equation (\ref{eq:13}) holds for $l=l^\ast$ and by the definition of $\Concentration=\concentration{\nbins}{\nballs}$ we obtain
\begin{align}
	\mu\left(l^\ast\right)=\Th{1}\frac{\nballs^{\Concentration+\delta}e^{\Concentration+\delta}}{\left(\Concentration+\delta\right)^{\Concentration+\delta+1/2}n^{\Concentration+\delta-1}}=\Th{1}\left(\frac{\nballs e}{\nbins\left(\Concentration+\delta\right)}\right)^\delta\left(\frac{\Concentration}{\Concentration+\delta}\right)^{\Concentration+1/2}.
\end{align}
Together with \Cref{lem:nu}\ref{lem:nu6} this yields  $\mu\left(l^\ast\right)=\smallo{1}$. Due to \Cref{lem:nu}\ref{lem:nu6} we have $\mu\left(l+1\right)/\mu\left(l\right)=\left(\nballs-l\right)/\left(\left(l+1\right)\left(\nbins-1\right)\right)=o(1)$ for all  $l\geq l^\ast$. Hence,
\begin{align*}
	\prob{\maxload\geq l^\ast}\lessorequal \sum_{l\geq l^\ast}\mu(l)=\left(1+\smallo{1}\right)\mu(l^\ast)=o(1).
\end{align*}
For a lower bound, we will show that $\prob{X^{\left(l_\ast\right)}>0}=1-o(1)$, where $l_\ast=l_\ast(\nbins):=\rounddown{\concentration{\nbins}{\nballs}-\varepsilon}$, using the second moment method. In the following we consider the random variables $X_j^{(l)}$ and $X^{(l)}$ only for $l=l_\ast$ and therefore we use $X_j=X_j^{\left(l_\ast\right)}$ and $X=X^{\left(l_\ast\right)}$ for simplicity. In order to apply the second moment method we will show $\expec{X}=\smallomega{1}$ and $\expec{X_i X_j}=\left(1+\smallo{1}\right)\expec{X_i}\expec{X_j}$ for all $i\neq j$. We let $\gamma=\gamma(n):=\Concentration-l_\ast\geq\varepsilon$ and by (\ref{eq:13}), \Cref{lem:nu}\ref{lem:nu6}, and the definition of $\Concentration$ we obtain
\begin{align*}
	\mu\left(l_\ast\right)=\Th{1}\frac{\nballs^{\Concentration-\gamma}e^{\Concentration-\gamma}}{\left(\Concentration-\gamma\right)^{\Concentration-\gamma+1/2}n^{\Concentration-\gamma-1}}=\Th{1}\left(\frac{\nbins\left(\Concentration+\delta\right)}{\nballs e}\right)^\gamma\left(\frac{\Concentration}{\Concentration-\gamma}\right)^{\Concentration+1/2}=\smallomega{1}.
\end{align*}
Next, we note that conditioned on the event $X_i=1$, i.e. $\load_i=l_\ast$, the loads $\load_j$ for $j\neq i$ are distributed like the loads of a balls--into--bins experiment with $\nbins-1$ bins and $\nballs-l_\ast$ balls, and thus
\begin{align*}
	\condprob{X_j=1}{X_i=1}\equal\binom{\nballs-l_\ast}{l_\ast}\left(1/(\nbins-1)\right)^{l_\ast}\left(1-1/(\nbins-1)\right)^{\nballs-2l_\ast}.
\end{align*}
Hence, we obtain
\begin{align*}
	\frac{\expec{X_i X_j}}{\expec{X_i}\expec{X_j}}&\equal\frac{\condprob{X_j=1}{X_i=1}}{\prob{X_j=1}}
	\\
	&\equal\frac{\binom{\nballs-l_\ast}{l_\ast}\left(1/(\nbins-1)\right)^{l_\ast}\left(1-1/(\nbins-1)\right)^{\nballs-2l_\ast}}{\binom{\nballs}{l_\ast}\left(1/n\right)^{l_\ast}\left(1-1/n\right)^{\nballs-l_\ast}}
	\\
	&\equal1+\smallo{1},
\end{align*}
where we used the assumption $\nballs>\nbins^{1/3}$ and the fact $l_\ast=\smallo{\log n}$ due to \Cref{lem:nu}\ref{lem:nu8}. Thus, by the second moment method we obtain $\prob{X>0}=1-o(1)$, which finishes the proof.
\end{proof}

%\begin{thm}[{\cite{balls_bins1}}]\label{thm:bins_balls_known}
%Let $\nballs=\nballs(n)=\Th{n}$, then \whp\ $\maxbinsandballs{n}{\nballs}=\left(1+\smallo{1}\right)\log n/\log \log n$.
%\end{thm}

%\begin{lem}\label{lem:stability}
%Let $d=d(n) \in \Z$ be such that $d=\smallo{n\log\log n/\log n}$. Then $\mutuallycontiguous{\maxbinsandballs{n}{n}}{\maxbinsandballs{n}{n+d}}$.
%\end{lem}

Next, we show that if we consider a \lq small\rq\ subset of bins, then the maximum load in one of these bins is significantly smaller than the maximum load of all bins. We will use this fact later when we relate random forests to the balls--into--bins model (see \Cref{sec:forests}), in which this \lq small\rq\ subset will correspond to the set of all roots.

\begin{lem}\label{lem:max_load_subset}
Let $\nballs=\nballs(n)$ and $t=t(n) \in \N$ be such that $\nballs=\Th{n}$ and $t= \smallo{n^{1-\delta}}$ for some $\delta>0$. Let $\location\sim\binsandballs{\nbins}{\nballs}$, $\maxload=\maxload(\location)$ be the maximum load, and $\maxload_{t}=\maxload_{t}(\location)$ be the maximum load in one of the first $t$ bins. Then, \whp\ $\maxload-\maxload_{ t}\equal\smallomega{1}$.
\end{lem}

\begin{proof}
We observe that $\maxload-\maxload_{ t}$ is strictly decreasing in $t$. Thus, it suffices to show $\maxload-\maxload_{ t}\equal\smallomega{1}$ for $t=\rounddown{n^{1-\delta}}$ and $\delta\in (0,1)$. We denote by $S_t$ the total number of balls in the first $t$ bins. We have $\expec{S_t}=t\nballs/\nbins$ and $\variance{S_t}\leq \expec{S_t}$. Hence, by Chebyshev's inequality, we have \whp\
\begin{align}\label{eq:19}
	S_t\lessorequal \frac{t\nballs}{\nbins}+\left(\frac{t\nballs}{\nbins}\right)^{2/3}=:\bar{l}=\bar{l}(n).
\end{align}
Conditioned on the event $S_t=l$ for $l\in\N$, $\maxload_{t}$ is distributed like $\maxbinsandballs{t}{l}$. Thus,
\begin{align*}
	\prob{\maxload_{t}\lessorequal \rounddown{\concentration{t}{\bar{l}}}+1}&\greaterorequal \sum_{l=1}^{\bar{l}}\probLarge{S_t=l}\condprob{\maxload_{t}\lessorequal \rounddown{\concentration{t}{\bar{l}}}+1}{S_t=l}
	\\
	&\greaterorequal \prob{S_t\leq \bar{l}}\prob{\maxbinsandballs{t}{\bar{l}}\lessorequal \rounddown{\concentration{t}{\bar{l}}}+1}
	\\
	&\equal\left(1-\smallo{1}\right),
\end{align*}
where the last equality follows from \Cref{thm:concentration_balls_bins} and (\ref{eq:19}). Due to \Cref{lem:nu}\ref{lem:nu1} and the assumption $t=\rounddown{n^{1-\delta}}$ we get $\concentration{t}{\bar{l}}=\left(1+\smallo{1}\right)\log t/\log \log t=\left(1-\delta+\smallo{1}\right)\log n/\log \log n$, which yields \whp\ $\maxload_{t}\leq\left(1-\delta+\smallo{1}\right)\log n/\log\log n$. By \Cref{lem:nu}\ref{lem:nu1} we have \whp\ $\maxload=\left(1+\smallo{1}\right)\log n/\log\log n$. Hence, we obtain \whp\ $\maxload-\maxload_{t}\greaterorequal \left(\delta+\smallo{1}\right)\log n/\log\log n=\smallomega{1}$, as desired.
\end{proof}

\section{\ER\ random graph and graphs without complex components}\label{sec:ergraph}
We start this section by providing a relation between the degree sequence of the \ER\ random graph $G(n,m)$ and the loads of a balls--into--bins model. In particular, this yields a refined version of \Cref{thm:max_degree_ergraph}. Later we will use that result to prove \Cref{thm:main_sub}, since the random planar graph $\planargraph(n,m)$ behaves similarly like $G(n,m)$ when $m\leq n/2+\bigo{n^{2/3}}$ (see \Cref{thm:non_complex}).

\begin{thm}\label{thm:G_n_m_bins_balls}
	Let $m=m(n)=\bigo{n}$ and $\degreesequence=\degreesequence(n)=\left(\degree{1}{G}, \ldots, \degree{n}{G}\right)$ be the degree sequence of $G=G(n,m)\ur \mathcal G(n,m)$. Moreover, let $\location=\location(n)\sim \binsandballs{n}{2m}$,  $\loadvector=\loadvector(n)=\loadvector(\location)$ be the vector of loads of $\location$, and $\varepsilon>0$. Then
	\begin{enumerate}
	\item\label{thm:G_n_m_bins_balls1}
	the degree sequence $\degreesequence$ is contiguous with respect to $\loadvector$, i.e.	$\contiguous{\degreesequence}{\loadvector}$;
	\item\label{thm:G_n_m_bins_balls2}
	\whp\ $\rounddown{\concentration{n}{2m}-\varepsilon}\lessorequal \maxdegree{G}\lessorequal \rounddown{\concentration{n}{2m}+\varepsilon}$.
	\end{enumerate}

\end{thm}
\begin{proof}
We consider the random multigraph $\multigraph$ given by $\vertexSet{\multigraph}=[n]$ and $\edgeSet{\multigraph}=\setbuilder{\left\{\locationBit_{2i-1}, \locationBit_{2i}\right\}}{i \in [m]}$, where $\location=\left(\locationBit_1, \ldots, \locationBit_{2m}\right)$ is the location vector. We observe that for $v\in[n]$ the load $\load_v$ equals the degree $\degree{v}{\multigraph}$. For each graph $H\in \mathcal{G}(n,m)$ we have $\prob{\multigraph=H}=2^mm!/n^{2m}$. Hence, conditioned on the event that $\multigraph$ is simple, $\multigraph$ is distributed like $G$. Moreover, for $n$ large enough we have \begin{align*}
	\prob{\multigraph \text{ is simple}}&\equal\prob{\multigraph \text{ has no loop}}\cdot\condprob{\multigraph \text{ has no multiple edge}}{\multigraph \text{ has no loop}}
	\\
	&\equal
	\left(1-\frac{1}{n}\right)^m\hspace{0.05cm}\cdot\hspace{0.05cm} \prod_{i=0}^{m-1}\left(1-\frac{i}{\binom{n}{2}}\right)
	\\
	&\greaterorequal
	\exp\left(\frac{-2m}{n}-\frac{4m^2}{n^2}\right)\greater\gamma,
\end{align*}
for a suitable chosen $\gamma>0$, since $m=\bigo{n}$. This shows $\liminf_{n\to \infty}\prob{\multigraph \text{ is simple}}>0$. Thus, each property that holds \whp\ in $\multigraph$, is also true \whp\ in $G$. In particular, the degree sequence $\degreesequence$ of $G$ is contiguous with respect to the degree sequence $\loadvector$ of $\multigraph$, i.e. $\contiguous{\degreesequence}{\loadvector}$. Together with \Cref{thm:concentration_balls_bins} this yields \whp\ $\rounddown{\concentration{n}{2m}-\varepsilon}\lessorequal \maxdegree{G}\lessorequal \rounddown{\concentration{n}{2m}+\varepsilon}$, as desired.
\end{proof}

We recall that we denote by $\nocomplex(n,m)$ a graph chosen uniformly at random from the class $\nocomplexclass(n,m)$ consisting of graphs having no complex components, vertex set $[n]$, and $m$ edges. Later $\nocomplex(n,m)$ will take the role of the non--complex part of the random planar graph. In this case the relation $m=n/2+\bigo{n^{2/3}}$ is satisfied (see \Cref{thm:internal_structure}). In \cite{uni} Britikov showed that in this range $\nocomplex(n,m)$ behaves similarly like $G(n,m)$. 

\begin{thm}[\cite{uni}]\label{thm:non_complex}
	Let $m=m(n)\leq n/2+\bigo{n^{2/3}}$ and $G=G(n,m)\ur \mathcal G(n,m)$ be the uniform random graph. Then 
	\begin{align*}
	\liminf_{n \to \infty}~ \prob{G \text{ has no complex component }}>0.
	\end{align*}
\end{thm}

Combining Theorems \ref{thm:G_n_m_bins_balls}\ref{thm:G_n_m_bins_balls2} and \ref{thm:non_complex} we can deduce that \whp\ $\maxdegree{\nocomplex(n,m)}$ is concentrated at two values.

\begin{lem}\label{lem:random_non_complex}
	Let $m=m(n)= n/2+\bigo{n^{2/3}}$, $\nocomplex=\nocomplex(n,m)\ur \nocomplexclass(n,m)$ be a random graph without complex components, and $\varepsilon>0$. Then \whp\ $\rounddown{\specialconcentration{n}-\varepsilon}\lessorequal \maxdegree{\nocomplex}\lessorequal \rounddown{\specialconcentration{n}+\varepsilon}$.
\end{lem}
\begin{proof}
Combining Theorems \ref{thm:G_n_m_bins_balls}\ref{thm:G_n_m_bins_balls2} and \ref{thm:non_complex} yields \whp
\begin{align}\label{eq:17}
	\rounddown{\concentration{n}{2m}-\varepsilon/2}\lessorequal \maxdegree{\nocomplex}\lessorequal \rounddown{\concentration{n}{2m}+\varepsilon/2}.
\end{align}
Using \Cref{lem:nu}\ref{lem:nu3} we obtain $\concentration{n}{2m}=\specialconcentration{n}+\smallo{1}$. Together with (\ref{eq:17}) this shows the statement. 
\end{proof}

\section{Random complex part and forests with specified roots}\label{sec:forests}
The goal of this section is to prove that \whp\ the maximum degree of a random complex part is concentrated at two values (see \Cref{thm:random_complex_part}\ref{thm:random_complex_part2}). As a random complex part can be constructed by using a random forest, we start by analysing the class $\forestclass(n,\ntrees)$ of forests on vertex set $[n]$ having $\ntrees$ trees (some of which might just be isolated vertices) such that the vertices $1, \ldots, \ntrees$ lie all in different trees. 

In \Cref{sub:pruefer} we generalise the concept of Prüfer sequences to forests. Then we use that idea to determine the maximum degree in a random forest in \Cref{sub:random_forest}. Finally, we state the concentration result on the maximum degree in a random complex part in \Cref{sub:random_complex_part}.

\subsection{Prüfer sequences for forests with specified roots}\label{sub:pruefer}
In the following we describe a bijection between $\forestclass(n, \ntrees)$ and $\sequences{n}{\ntrees}:=[n]^{n-\ntrees-1}\times [\ntrees]$, similar to the Prüfer sequence for trees (see e.g. \cite{book_pruefer, book_pruefer2}). Given a forest $\forest \in \forestclass(n,\ntrees)$ we construct a sequence $\left(\forest_0, \ldots, \forest_{n-\ntrees}\right)$ of forests and two sequences $\left(x_1, \ldots, x_{n-\ntrees}\right)$ and $\left(y_1, \ldots, y_{n-\ntrees}\right)$ of vertices as follows. We start with $\forest_0:=\forest$. Given $\forest_{i-1}$ for an $i\in[n-\ntrees]$, we let $y_i$ be the leaf with largest label in $\forest_{i-1}$ and $x_i$ be the unique neighbour of $y_i$. Furthermore, we obtain $\forest_i$ by deleting the edge $x_iy_i$ in $\forest_{i-1}$. We note that this construction is always possible, since $\forest_{i-1}$ has $n-t-i+1$ edges and therefore at least one leaf. We call 
\begin{align}\label{eq:20}
\prueferseqence(\forest):=\left(x_1, \ldots, x_{n-\ntrees}\right)
\end{align}
the {\em Prüfer sequence} of $\forest$. We will show that $\prueferseqence$ is a bijection between $\forestclass(n, \ntrees)$ and $\sequences{n}{\ntrees}$. In addition, we will prove that the number of occurrences of a vertex $v\in[n]$ in the Prüfer sequences is determined by the degree of the vertex. To that end, given $v\in[n]$ and $\mathbf{w}=\left(w_1, \ldots, w_{n-\ntrees}\right)\in[n]^{n-\ntrees}$ we denote by $\frequency{v}{\mathbf{w}}:=\left|\setbuilder{i\in[n-\ntrees]}{w_i=v}\right|$ the number of occurrences of $v$ in $\mathbf{w}$.
\begin{thm}\label{thm:pruefer}
Let $n, \ntrees\in \N$ and $\forestclass(n,\ntrees)$ be the class of forests on vertex set $[n]$ consisting of $\ntrees$ trees such that the vertices $1, \ldots, \ntrees$ lie all in different trees. In addition, let $\sequences{n}{\ntrees}=[n]^{n-\ntrees-1}\times [\ntrees]$ and $\prueferseqence(\forest)$ be the Prüfer sequence of $F\in\forestclass(n,\ntrees)$ as defined in (\ref{eq:20}). Then $\prueferseqence:\forestclass(n,\ntrees)\to \sequences{n}{\ntrees}$ is a bijection and for $\forest\in \forestclass(n,\ntrees)$ and $v\in[n]$ we have \begin{align}\label{eq:12}
	\degree{v}{\forest}\equal
	\begin{cases}
		\frequency{v}{\prueferseqence(\forest)} & \text{~~if~~} v\in [\ntrees]
		\\
		\frequency{v}{\prueferseqence(\forest)}+1 & \text{~~if~~} v\in[n]\setminus [\ntrees].
	\end{cases}
\end{align}
\end{thm}

\begin{proof}
We start by proving (\ref{eq:12}). To that end, let $r\in[\ntrees]$ be a root vertex. Throughout the construction of $\prueferseqence(\forest)$ the root $r$ is always the vertex with smallest label in the component of $\forest_{i}$ containing $r$. This implies $r\neq y_i$ for all $i\in[n-\ntrees]$. As the elements of the sequence $\mathbf{y}=\left(y_1, \ldots, y_{n-\ntrees}\right)$ are all distinct, we obtain 
\begin{align}\label{eq:21}
	\frequency{v}{\mathbf{y}}=
	\begin{cases}
		0 & \text{~~if~~} v\in [\ntrees]
		\\
		1 & \text{~~if~~} v\in[n]\setminus [\ntrees].
	\end{cases}
\end{align}
This proves (\ref{eq:12}), since $\degree{v}{\forest}=\frequency{v}{\mathbf{x}}+\frequency{v}{\mathbf{y}}$.

Next, we provide an algorithm that builds a graph $\prueferinvers(\mathbf{w})$ for each $\mathbf{w}\in \sequences{n}{\ntrees}$. Later we will see that the algorithm indeed reconstructs $\forest\in\forestclass(n,\ntrees)$ if the input is $\mathbf{w}=\prueferseqence(\forest)$. Let $\mathbf{w}=\left(w_1, \ldots, w_{n-\ntrees}\right)\in \sequences{n}{\ntrees}$ be given. We construct sequences $(\tilde{x}_1, \ldots, \tilde{x}_{n-\ntrees})$ and $(\tilde{y}_1, \ldots, \tilde{y}_{n-\ntrees})$ of vertices, a sequence $\left(\tilde{\forest}_0, \ldots, \tilde{\forest}_{n-\ntrees}\right)$ of forests and for each $v\in[n]$ a sequence $(\tilde{d}_0(v), \ldots, \tilde{d}_{n-\ntrees}(v))$ of degrees as follows. We start with $\vertexSet{\tilde{\forest}_0}=[n]$, $\edgeSet{\tilde{\forest}_0}=\emptyset$, $\tilde{d}_0(v)=\frequency{v}{\mathbf{w}}$ if $v\in [\ntrees]$, and $\tilde{d}_0(v)=\frequency{v}{\mathbf{w}}+1$ if $v\in [n]\setminus[\ntrees]$. For $i\in[n-\ntrees]$ we set $\tilde{x}_i=w_i$ and $\tilde{y}_i=\max\setbuilder{v}{\tilde{d}_{i-1}(v)=1}$. In addition, we let $\tilde{d}_i(v)=\tilde{d}_{i-1}(v)-1$ if $v\in\{\tilde{x}_i, \tilde{y}_i\}$ and $\tilde{d}_i(v)=\tilde{d}_{i-1}(v)$ otherwise. Finally, we obtain $\tilde{\forest}_i$ by adding the edge $\tilde{x}_i\tilde{y}_i$ in $\tilde{\forest}_{i-1}$ and we set $\prueferinvers(\mathbf{w}):=\tilde{\forest}_{n-\ntrees}$. Next, we show that this algorithm is well--defined and that the output is indeed a graph. To that end, we note that for $v\in[n]\setminus[\ntrees]$ and $i\in[n-\ntrees]$ we have
\begin{align*}
	\big(\tilde{d}_{i-1}(v)\geq 1, \tilde{d}_i(v)=0\big)~ \implies~ \big(\tilde{y}_i=v\big).
\end{align*}
This yields that there are at least $(n-\ntrees-i)$ vertices $v\in[n]\setminus[\ntrees]$ with $\tilde{d}_i(v)\geq 1$. Thus, if $\tilde{d}_{i-1}(w_{n-\ntrees})\geq 1$ for some $i\in[n-\ntrees]$, then $\sum_{v\in [n]\setminus [\ntrees]}\tilde{d}_{i-1}(v)\leq 2(n-\ntrees-i)+1$ and therefore $\tilde{y}_i\in[n]\setminus[\ntrees]$. This yields $\tilde{d}_{i}(w_{n-\ntrees})\geq 1$ unless $i=n-\ntrees$. As $\tilde{d}_{0}(w_{n-\ntrees})\geq 1$ we obtain by induction that $\tilde{y}_i\in[n]\setminus[\ntrees]$ for all $i\in[n-\ntrees]$. In particular, this shows that $\tilde{y}_i$ is well--defined and $\tilde{x}_i\neq \tilde{y}_i$. Thus, the algorithm is always executable and $\tilde{\forest}_{i}$ is a graph for all $i\in[n-\ntrees]$.

In order to prove that $\prueferseqence:\forestclass(n,\ntrees)\to \sequences{n}{\ntrees}$ is a bijection, it suffices to show the following claims.
\begin{enumerate}[label=(\roman*)]
	\item \label{cl:pruefer1}
	$\prueferseqence(\forest)\in \sequences{n}{\ntrees}$ for all $\forest\in \forestclass(n, \ntrees)$;
	\item \label{cl:pruefer2} $\prueferinvers(\prueferseqence(\forest))=\forest$ for all $\forest\in \forestclass(n, \ntrees)$;
	\item \label{cl:pruefer3}
	$\prueferinvers(\mathbf{w})\in\forestclass\left(n, \ntrees\right)$ for all $\mathbf{w}\in \sequences{n}{\ntrees}$;
	\item \label{cl:pruefer4}
	$\prueferseqence\left(\prueferinvers\left(\mathbf{w}\right)\right)=\mathbf{w}$ for all $\mathbf{w}\in \sequences{n}{\ntrees}$.
\end{enumerate}
We observe that $x_{n-\ntrees}\notin \left\{y_1, \ldots, y_{n-\ntrees}\right\}$. Thus, using (\ref{eq:21}) yields $x_{n-\ntrees}\in[\ntrees]$, which implies \ref{cl:pruefer1}.

To show \ref{cl:pruefer2} we suppose that we first apply the algorithm to obtain $\prueferseqence(\forest)$ and then the algorithm $\prueferinvers$ with input $\mathbf{w}=\prueferseqence(\forest)$. Due to (\ref{eq:12}) the degree sequence of $\forest_0=\forest$ equals $\left(\tilde{d}_0(1), \ldots, \tilde{d}_0(n)\right)$ and therefore $\tilde{y}_1=y_1$. By construction we also have $\tilde{x}_1=x_1$, which implies that $\left(\tilde{d}_1(1), \ldots, \tilde{d}_1(n)\right)$ is the degree sequence of $\forest_1$. By repeating that argument we obtain by induction $\tilde{y}_i=y_i$ for all $i\in[n-\ntrees]$. As $\edgeSet{\forest}=\setbuilder{x_iy_i}{i\in[n-\ntrees]}$ and $\edgeSet{\tilde{\forest}_{n-\ntrees}}=\setbuilder{\tilde{x}_i\tilde{y}_i}{i\in[n-\ntrees]}$ this shows $\tilde{\forest}_{n-\ntrees}=\forest$, i.e. $\prueferinvers(\prueferseqence(\forest))=\forest$.

For \ref{cl:pruefer3} we assume that we apply the algorithm $\prueferinvers$ with input $\mathbf{w}\in \sequences{n}{\ntrees}$.  By induction it follows that for all $i\in\left\{0, \ldots, n-\ntrees\right\}$ each component of $\tilde{\forest}_i$ contains at most one vertex $v$ with $\tilde{d}_i(v)>0$. This implies that we never close a cycle when adding the edge $\tilde{x}_{i+1}\tilde{y}_{i+1}$ in $\tilde{\forest}_i$, which shows that $\prueferinvers(\mathbf{w})$ is a forest. We saw before that $\tilde{y}_i\in[n]\setminus[\ntrees]$ for all $i\in[n-\ntrees]$. Thus, if $r\in[\ntrees]$ is a root and the component of $\tilde{\forest}_i$ containing $r$ has a vertex $v$ with $\tilde{d}_i(v)>0$, then $v=r$. This implies that adding the edge $\tilde{x}_{i+1}\tilde{y}_{i+1}$ never connects two components of $\tilde{\forest}_i$ which contain both a root. Hence, $\prueferinvers(\mathbf{w})\in \forestclass(n,\ntrees)$. 

Finally for \ref{cl:pruefer4}, we suppose that for given $\mathbf{w}\in \sequences{n}{\ntrees}$ we first apply the algorithm to construct $\prueferinvers(\mathbf{w})$ and then the algorithm to obtain the Prüfer sequence of $\prueferinvers(\mathbf{w})$. We note that the degree sequence of $\forest_0=\prueferinvers(\mathbf{w})$ equals $\left(\tilde{d}_0(1), \ldots, \tilde{d}_0(n)\right)$ and therefore $y_1=\tilde{y}_1$. By construction $\tilde{x}_1$ is the unique neighbour of $\tilde{y}_1$ in $\forest_0$, which implies $x_1=\tilde{x}_1$. This yields that the degree sequence of $\forest_1$ is $\left(\tilde{d}_1(1), \ldots, \tilde{d}_1(n)\right)$. Repeating that argument we obtain by induction $\tilde{x}_i=x_i$ for all $i\in[n-\ntrees]$, which proves \ref{cl:pruefer4}.
\end{proof}

\subsection{Random forests with specified roots}\label{sub:random_forest}
In this section we consider a random forest $\forest=\forest(n,\ntrees)\ur \forestclass(n,\ntrees)$. In particular, we are interested in the degree sequence of $\forest$ and the maximum degree $\maxdegree{\forest}$. Instead of directly choosing $\forest$ from $\forestclass(n,\ntrees)$, we can equivalently create $\forest$ by Prüfer sequences from \Cref{sub:pruefer}: First we perform a balls--into--bins experiment with $n$ bins and $n-\ntrees-1$ balls and let $\location=\left(\locationBit_1, \ldots, \locationBit_{n-\ntrees-1}\right)\sim\binsandballs{n}{n-\ntrees-1}$ be the location vector. Then we independently choose $\locationBit_{n-\ntrees}\ur [\ntrees]$ and set $\forest=\prueferinvers\left(\locationBit_1, \ldots, \locationBit_{n-\ntrees}\right)$. We note that the degrees $\degree{v}{\forest}$ in $\forest$ are given by (\ref{eq:12}). Thus, we obtain the following.

\begin{thm}\label{thm:forest_balls_bins}
Let $n, \ntrees\in \N$ and $\degreesequence=\left(\degree{1}{\forest}, \ldots, \degree{n}{\forest}\right)$ be the degree sequence of $\forest=\forest\left(n, \ntrees\right)\ur \forestclass\left(n, \ntrees\right)$. Let $\location\sim \binsandballs{n}{n-\ntrees-1}$ and $\load_j=\load_j(\location)$ be the load in bin $\bin_j$ for each $j\in[n]$. In addition, let $\rootrv\ur [\ntrees]$ (independent of $\forest$) and for $j\in[\ntrees]$ we define $\rootvectorbit_j=1$ if $\rootrv=j$ and $\rootvectorbit_j=0$ otherwise. Then
\begin{align*}
\big(\load_1+\rootvectorbit_1, \ldots, \load_\ntrees+\rootvectorbit_\ntrees, \load_{\ntrees+1}+1, \ldots, \load_n+1\big)\sim \degreesequence.
\end{align*}
\end{thm}

\begin{proof}
The statement follows by combining \Cref{thm:pruefer} and the construction of $\forest$ via Prüfer sequences as described above.
\end{proof}

Using this connection to the balls--into--bins model we obtain an upper bound on $\maxdegree{\forest(n,\ntrees)}$ (see \Cref{thm:forest_max_degree}\ref{thm:forest_max_degree1}). If we assume that $\ntrees$ is not too \lq large\rq, we can even show that \whp\ $\maxdegree{\forest(n,\ntrees)}$ is concentrated at two values and that the maximum degree of a root vertex, i.e. a vertex in $[\ntrees]$, is much smaller than $\maxdegree{\forest(n,\ntrees)}$ (see \Cref{thm:forest_max_degree}\ref{thm:forest_max_degree2}). We will need these facts later when we use random forests to build a random complex part (see \Cref{sub:random_complex_part}).
\begin{thm}\label{thm:forest_max_degree}
Let $\ntrees=\ntrees(n)$, $\forest=\forest\left(n, \ntrees\right)\ur \forestclass\left(n, \ntrees\right)$, and $\varepsilon>0$. Then 
\begin{enumerate}
\item\label{thm:forest_max_degree1}
\whp\
$\maxdegree{\forest}\lessorequal \rounddown{\specialconcentration{n}}+2$;
\item\label{thm:forest_max_degree2}
if $\ntrees=\smallo{n^{1-\delta}}$ for some $\delta>0$, then \whp\ $\rounddown{\specialconcentration{\nbins}-\varepsilon}+1\lessorequal \maxdegree{\forest}\lessorequal \rounddown{\specialconcentration{\nbins}+\varepsilon}+1$ and $\maxdegree{\forest}-\max\setbuilder{d_\forest(r)}{r \in [\ntrees]}=\smallomega{1}$.
\end{enumerate}
\end{thm}

\begin{proof}
Let $\location\sim \binsandballs{n}{n-\ntrees-1}$, $\maxload=\maxload(\location)$ be the maximum load of $\location$, and $\maxload_{t}=\maxload_{t}(\location)$ be the maximum load of one of the first $t$ bins of $\location$. Due to \Cref{thm:concentration_balls_bins} we have \whp\ $\maxload\leq \rounddown{\concentration{n}{n-\ntrees-1}}+1\leq\rounddown{\specialconcentration{n}}+1$, where we used in the last inequality \Cref{lem:nu}\ref{lem:nu4}. Combining it with \Cref{thm:forest_balls_bins} we have 
\begin{align*}
\probLarge{\maxdegree{\forest}> \rounddown{\specialconcentration{n}}+2}\lessorequal \prob{\maxload >\rounddown{\specialconcentration{n}}+1}=\smallo{1}.
\end{align*}
This shows statement \ref{thm:forest_max_degree1}.

By \Cref{lem:max_load_subset} we have \whp\ $\maxload-\maxload_{t}=\smallomega{1}$. This together with \Cref{thm:forest_balls_bins} implies \whp\  $\maxdegree{\forest}-\max\setbuilder{d_\forest(r)}{r \in [\ntrees]}=\smallomega{1}$ and $\contiguous{\maxdegree{\forest}}{\maxload+1}$. Thus, we obtain by \Cref{thm:concentration_balls_bins} that \whp
\begin{align*}
	\rounddown{\concentration{n}{n-\ntrees-1}-\varepsilon/2}+1\lessorequal \maxdegree{\forest}\lessorequal \rounddown{\concentration{n}{n-\ntrees-1}+\varepsilon/2}+1.
\end{align*}
By \Cref{lem:nu}\ref{lem:nu3} we have $\concentration{n}{n-\ntrees-1}=\specialconcentration{n}+\smallo{1}$, which shows statement \ref{thm:forest_max_degree2}.
\end{proof}

We note that the special case of random trees, i.e. when $\ntrees=1$, was studied in \cite{moon,CGS}. In particular, Carr, Goh, and Schmutz \cite{CGS} used the saddle--point method to show that \whp\ the maximum degree in random trees is concentrated at two values.

%\subsection{Random trees}
%Next, we consider the special case $\ntrees=1$. We observe that the elements in $\forestclass(n,1)$ are trees on vertex set $[n]$, i.e. $\forestclass(n,1)=\treeclass(n)$. Carr, Goh, and Schmutz \cite{CGS} showed that the maximum degree $\maxdegree{\tree}$ of a random tree $\tree=\tree(n)\ur\treeclass(n)$ is strongly concentrated. More precisely, they proved the following statement.
%
%\begin{thm}[\cite{CGS}]\label{thm:max_degree_tree_concentration}
%	Let $\tree=\tree(n)\ur\treeclass(n)$ and $l=l(n)\in \N$ such that $l=\left(1+\smallo{1}\right)\log n/\log\log n$. Then
%	\begin{align*}
%	\prob{\maxdegree{\tree}\leq l}=\exp\left[-\exp\left[\log n-l\log l+l-1/2\log l-\log\left(e\sqrt{2\pi}\right)+o(1)\right]\right]+o(1).
%	\end{align*}
%\end{thm}
%
%Combining \Cref{thm:forest_balls_bins} and \Cref{thm:max_degree_tree_concentration} yields the following strong concentration result on $\maxbinsandballs{n}{n}$.
%
%\begin{coro}\label{coro:concentration_bins_balls}
%	Let $\varepsilon>0$, then \whp
%	\begin{align*}
%	\roundup{\sol(n)-\varepsilon}-1\leq \maxbinsandballs{n}{n}\leq \roundup{\sol(n)+\varepsilon}-1.
%	\end{align*}
%\end{coro}

\subsection{Random complex part}\label{sub:random_complex_part}
In this section we consider the class $\complexclass\left(C,q\right)$ consisting of complex graphs with core $C$ and vertex set $[q]$, where $C$ is a given core and $q\in \N$ (cf. \Cref{def:random_complex_part}). Provided that $\maxdegree{C}$ is bounded and $\numberVertices{C}$ is \lq small\rq\ compared to $q$, we can use \Cref{thm:forest_max_degree} to show that the maximum degree of $\complexgraph\left(C,q\right)\ur \complexclass\left(C,q\right)$ is strongly concentrated.

\begin{thm}\label{thm:random_complex_part}
For each $n\in \N$, let $C=C(n)$ be a core and $q=q(n)\in \N$. In addition, let $\complexgraph=\complexgraph\left(C,q\right)\ur \complexclass\left(C,q\right)$ be a random complex part with core $C$ and vertex set $[q]$ as in \Cref{def:random_complex_part} and $\varepsilon>0$. If $\maxdegree{C}=\Th{1}$, then the following hold.
\begin{enumerate}
\item\label{thm:random_complex_part1}
\Whp\ $\maxdegree{Q}\lessorequal \specialconcentration{q}+\bigo{1}$.
\item\label{thm:random_complex_part2}
If in addition $\numberVertices{C}=\smallo{q^{1-\delta}}$ for some $\delta>0$, then \whp\ $\rounddown{\specialconcentration{q}-\varepsilon}+1\lessorequal \maxdegree{Q}\lessorequal \rounddown{\specialconcentration{q}+\varepsilon}+1$.
\end{enumerate}
\end{thm}

\begin{proof}
We observe that $Q$ can be obtained by choosing a random forest $\forest=\forest(q,\numberVertices{C})\ur\forestclass(q,\numberVertices{C})$ and then replacing each vertex $r$ in $C$ by the tree of $\forest$ with root $r$. For a vertex $v\in \vertexSet{Q}$ we have
\begin{align}\label{eq:3}
	\degree{v}{Q}=
	\begin{cases}
		\degree{v}{C}+\degree{v}{F}& \text{if}~~ v \in \vertexSet{C} \\
		\degree{v}{F} & \text{ otherwise}.
	\end{cases}
\end{align}
Hence, we have $\maxdegree{Q}\lessorequal \maxdegree{C}+\maxdegree{\forest}$. By \Cref{thm:forest_max_degree}\ref{thm:forest_max_degree1} we get \whp\ $\maxdegree{\forest}\lessorequal \specialconcentration{q}+2$. Together with the fact $\maxdegree{C}=\Th{1}$ this yields statement \ref{thm:random_complex_part1}. For \ref{thm:random_complex_part2} we apply \Cref{thm:forest_max_degree}\ref{thm:forest_max_degree2} to $\forest$. Together with (\ref{eq:3}) and $\maxdegree{C}=\Th{1}$ this implies \whp\ $\maxdegree{Q}=\maxdegree{\forest}$. Thus, statement \ref{thm:random_complex_part2} follows by applying again \Cref{thm:forest_max_degree}\ref{thm:forest_max_degree2}.
\end{proof}

\section{Proofs of main results}\label{sec:proof}
Throughout this section, let $\planargraph=\planargraph(n,m)\ur \planarclass(n,m)$ be the random planar graph.
\proofof{thm:main_sub}
In \Cref{thm:non_complex} we have seen that $\liminf_{n \to \infty} \prob{G(n,m)\text{ is planar}}>0$. Thus, each graph property that holds \whp\ in $G(n,m)$ is also true \whp\ in $\planargraph$ and the first statement follows by \Cref{thm:G_n_m_bins_balls}\ref{thm:G_n_m_bins_balls2}. By taking $\varepsilon=1/3$ we get the \lq in particular\rq\ statement. \qed

\proofof{thm:main_1sup}
We split $\planargraph$ into the large complex part $\Complexlargestcore=\complexlargestcore{\planargraph}$, the small complex part $\Complexrestcore=\complexrestcore{\planargraph}$, and the non--complex part $\Restcomplex=\restcomplex{\planargraph}$ as described in \Cref{sub:decomposition}. We claim that \whp\ the following hold.
\begin{enumerate}[label=(\roman*)]
\item\label{cl:1}
$\rounddown{\specialconcentration{s}-\varepsilon}+1\lessorequal\maxdegree{\Complexlargestcore}\lessorequal\rounddown{\specialconcentration{s}+\varepsilon}+1$;
\item\label{cl:2}
$\maxdegree{\Complexrestcore}\lessorequal\specialconcentration{n^{2/3}}+\bigo{1}$;
\item\label{cl:3}
$\rounddown{\specialconcentration{n}-\varepsilon}\lessorequal\maxdegree{\Restcomplex}\lessorequal\rounddown{\specialconcentration{n}+\varepsilon}$.
\end{enumerate}
Assuming these three claims are true we can finish the proof as follows. By \Cref{thm:internal_structure} we have \whp\ $\Largestcomponent=\Complexlargestcore$ and therefore also \whp\ $\Rest=\Complexrestcore\cup \Restcomplex$. Thus, the
statement \ref{thm:main_1sup1} of \Cref{thm:main_1sup} follows by \ref{cl:1}. By \Cref{lem:nu}\ref{lem:nu1} we have $\specialconcentration{n^{2/3}}=\left(2/3+\smallo{1}\right)\log n/\log\log n$ and $\specialconcentration{n}=\left(1+\smallo{1}\right)\log n/\log\log n$. Combining that with \ref{cl:2} and \ref{cl:3} yields \whp\ $\maxdegree{\Complexrestcore\cup \Restcomplex}=\maxdegree{\Restcomplex}$ and therefore also \whp\ $\maxdegree{\Rest}=\maxdegree{\Restcomplex}$. Hence, the statement \ref{thm:main_1sup2} of \Cref{thm:main_1sup} follows by \ref{cl:3}. Finally, we obtain the \lq in particular\rq\ statement by taking $\varepsilon=1/3$.

To prove the claims, we will follow the strategy described in \Cref{sub:strategy_decomposition}: We will construct a conditional random graph $\condGraph{\randomGraph}{\seq}$ which is distributed like the random graph $\tilde{\planargraph}$ introduced in \Cref{sub:strategy_decomposition}. Then we will determine the maximum degrees of the large complex part, small complex part and non--complex part of $\condGraph{\randomGraph}{\seq}$ (or equivalently of $\tilde{\planargraph}$). Finally, we will apply \Cref{lem:conditional_random_graphs} to translate these results to the random planar graph $\planargraph$. 

Let $\cl(n)$ be the subclass of $\planarclass(n,m)$ consisting of those graphs $H$ satisfying
\begin{align}
\maxdegreeLarge{\core{H}}&\equal3, \label{eq:7}\\
\numberVerticesLarge{\largestcomponent{\core{H}}}&\equal\Th{sn^{-1/3}},\label{eq:8}\\
\numberVerticesLarge{\complexlargestcore{H}}&\equal\left(2+\smallo{1}\right)s,\label{eq:9}\\
\numberVerticesLarge{\complexrestcore{H}}&\equal\bigo{n^{2/3}},\label{eq:10}\\
\numberEdgesLarge{\restcomplex{H}}&\equal\numberVerticesLarge{\restcomplex{H}}/2+\bigo{\numberVerticesLarge{\restcomplex{H}}^{2/3}}.\label{eq:11}
\end{align}
Due to \Cref{thm:internal_structure} we can choose the implicit hidden constants in the equations (\ref{eq:7})--(\ref{eq:11}) such that $\planargraph\in\cl(n)$ with a probability of at least $1-\gamma/2$, for arbitrary $\gamma>0$. We will apply \Cref{lem:conditional_random_graphs} to the class $\cl:=\bigcup_{n\in\N}\cl(n)$. To that end, we define the function $\func$ such that for $H\in\cl$ we have
\begin{align*}
\func(H):=\big(\core{H}, \numberVertices{\complexlargestcore{H}}, \numberVertices{\complexrestcore{H}}\big).
\end{align*}
Let $\seq=\left(C_n, l_n, r_n\right)_{n\in\N}$ be a sequence that is feasible for $(\cl, \func)$ and let $\randomGraph=\randomGraph(n)\ur \cl(n)$. 
By definition the possible realisations of $\condGraph{\randomGraph}{\seq}$ are those graphs $H \in \cl$ with  $\core{H}=C_n$, $\numberVertices{\complexlargestcore{H}}=l_n$, and $\numberVertices{\complexrestcore{H}}=r_n$. Hence, $\condGraph{\randomGraph}{\seq}=\complexlargestcore{\condGraph{\randomGraph}{\seq}}~\dot\cup~ \complexrestcore{\condGraph{\randomGraph}{\seq}}~\dot\cup~ \restcomplex{\condGraph{\randomGraph}{\seq}}$ can be constructed as follows. For $\complexlargestcore{\condGraph{\randomGraph}{\seq}}$ we choose uniformly at random a complex graph with $l_n$ vertices and core $\largestcomponent{C_n}$ and for $\complexrestcore{\condGraph{\randomGraph}{\seq}}$ a complex graph with $r_n$ vertices and core $\rest{C_n}$. For $\restcomplex{\condGraph{\randomGraph}{\seq}}$ we choose a graph without complex components having $u_n:=n-l_n-r_n$ vertices and $w_n:=m-\numberEdges{C_n}+\numberVertices{C_n}-l_n-r_n$ edges. Summing up, we have
\begin{align}
\complexlargestcoreLarge{\condGraph{\randomGraph}{\seq}} &\sim \complexgraph\big(\largestcomponent{C_n}, l_n\big);\label{eq:14}\\
\complexrestcoreLarge{\condGraph{\randomGraph}{\seq}} &\sim \complexgraph\big(\rest{C_n}, r_n\big);\label{eq:15}\\
\restcomplexLarge{\condGraph{\randomGraph}{\seq}}&\sim \nocomplex\big(u_n,w_n\big),\label{eq:16}
\end{align}
where the random complex parts and the random graph without complex components on the right hand side are as defined in \Cref{def:random_complex_part,def:nocomplex}, respectively.
Due to (\ref{eq:7})--(\ref{eq:9}) we have $\maxdegree{C_n}=3$, $\numberVerticesLarge{\largestcomponent{C_n}}=\Th{sn^{-1/3}}$, and $l_n=\left(2+\smallo{1}\right)s$. Hence, we can apply \Cref{thm:random_complex_part}\ref{thm:random_complex_part2} to $\complexgraph\big(\largestcomponent{C_n}, l_n\big)$. Together with (\ref{eq:14}) this implies \whp
\begin{align}\label{eq:4}
	\rounddown{\specialconcentration{l_n}-\varepsilon/2}+1\lessorequal \maxdegreeLarge{\complexlargestcore{\condGraph{\randomGraph}{\seq}}}\lessorequal 	\rounddown{\specialconcentration{l_n}+\varepsilon/2}+1.
\end{align}
By \Cref{lem:nu}\ref{lem:nu2} we have $\specialconcentration{l_n}=\specialconcentration{s}+\smallo{1}$, since $l_n=\left(2+\smallo{1}\right)s$. Together with (\ref{eq:4}) this shows \whp
\begin{align}\label{eq:5}
	\rounddown{\specialconcentration{s}-\varepsilon}+1\lessorequal \maxdegreeLarge{\complexlargestcore{\condGraph{\randomGraph}{\seq}}}\lessorequal 	\rounddown{\specialconcentration{s}+\varepsilon}+1.
\end{align}
By \Cref{lem:conditional_random_graphs}, inequality (\ref{eq:5}) is also \whp\ true if we replace $\condGraph{\randomGraph}{\seq}$ by $\randomGraph$. Combining that with the fact that $\planargraph \in \cl$ with probability at least $1-\gamma/2$ we obtain that with probability at least $1-\gamma$
\begin{align*}
	\rounddown{\specialconcentration{s}-\varepsilon}+1\lessorequal \maxdegree{\Complexlargestcore}\lessorequal 	\rounddown{\specialconcentration{s}+\varepsilon}+1
\end{align*}
for all $n$ large enough. As $\gamma$ was arbitrary, this shows claim \ref{cl:1}. 

Next, we prove claims \ref{cl:2} and \ref{cl:3} in a similar fashion. Combining (\ref{eq:15}) and \Cref{thm:random_complex_part}\ref{thm:random_complex_part1} yields  $\maxdegreeLarge{\complexrestcore{\condGraph{\randomGraph}{\seq}}}\leq \specialconcentration{r_n}+\bigo{1}$. Due to \Cref{lem:nu}\ref{lem:nu9} and \ref{lem:nu2} we have $\specialconcentration{r_n}\leq\specialconcentration{n^{2/3}}+o(1)$, where we used $r_n=\bigo{n^{2/3}}$ by (\ref{eq:10}). This yields $\maxdegreeLarge{\complexrestcore{\condGraph{\randomGraph}{\seq}}}\leq\specialconcentration{n^{2/3}}+\bigo{1}$. Thus, claim \ref{cl:2} follows by \Cref{lem:conditional_random_graphs}. 
Due to (\ref{eq:11}) we have $w_n=u_n/2+\bigo{u_n^{2/3}}$. Hence, we can combine (\ref{eq:16}) and \Cref{lem:random_non_complex} to obtain \whp
\begin{align}\label{eq:6}
\rounddown{\specialconcentration{u_n}-\varepsilon/2}\lessorequal \maxdegreeLarge{\restcomplex{\condGraph{\randomGraph}{\seq}}}\lessorequal \rounddown{\specialconcentration{u_n}+\varepsilon/2}.
\end{align}
Due to (\ref{eq:9}) and (\ref{eq:10}) we have $u_n=n-l_n-r_n=\left(1-\smallo{1}\right)n$ and therefore we obtain  $\specialconcentration{u_n}=\specialconcentration{n}+\smallo{1}$ by \Cref{lem:nu}\ref{lem:nu2}. Using that in (\ref{eq:6}) we get \whp
 \begin{align*}
 \rounddown{\specialconcentration{n}-\varepsilon}\lessorequal \maxdegreeLarge{\restcomplex{\condGraph{\randomGraph}{\seq}}}\lessorequal \rounddown{\specialconcentration{n}+\varepsilon}.
 \end{align*}
 Now claim \ref{cl:3} follows by \Cref{lem:conditional_random_graphs}.  \qed
\proofof{thm:main_int}
The statement follows analogously to the proof of \Cref{thm:main_1sup}. \qed

\proofof{cor:maxdegree}
We distinguish three cases according to which region the number of edges $m$ falls into. If $m$ is as in \Cref{thm:main_sub}, then by \Cref{lem:nu}\ref{lem:nu1} \whp\ $\maxdegree{\planargraph}=\concentration{n}{2m}+\bigo{1}=\left(1+\smallo{1}\right)\log n/\log\log n$, where we used that $m=\Th{n}$. Similarly, if $m$ is as in \Cref{thm:main_1sup}, then due to \Cref{lem:nu} \ref{lem:nu1} and \ref{lem:nu9} we have \whp\ $\maxdegree{\planargraph}=\specialconcentration{n}+\bigo{1}=\left(1+\smallo{1}\right)\log n/\log\log n$. Finally, if $m$ is as in \Cref{thm:main_int}, then by \Cref{lem:nu}\ref{lem:nu1} \whp\ $\maxdegree{\planargraph}=\specialconcentration{n}+\bigo{1}=\left(1+\smallo{1}\right)\log n/\log\log n$.
\qed

\section{Discussion}\label{sec:discussion}
The only properties about $\planargraph(n,m)$ which we used in our proofs are the results on the internal structure in \Cref{thm:internal_structure}. Kang, Moßhammer, and Sprüssel \cite{surface} showed that \Cref{thm:internal_structure} is true for much more general classes of graphs. Prominent examples of such classes are cactus graphs, series--parallel graphs, and graphs embeddable on an orientable surface of genus $g\in \N\cup \{0\}$ (see \cite[Section 4]{cycles}). Using the generalised version of \Cref{thm:internal_structure} and analogous proofs of \Cref{thm:main_sub,thm:main_1sup,thm:main_int} and \Cref{cor:maxdegree}, one can show the following.
\begin{thm}
\Cref{thm:main_sub,thm:main_1sup,thm:main_int} and \Cref{cor:maxdegree} are true for the class of cactus graphs, the class of series--parallel graphs, and the class of graphs embeddable on an orientable surface of genus $g\in \N \cup \{0\}$.
\end{thm}

\bibliographystyle{plain}
\bibliography{kang-missethan-max-degree}

\appendix
\section{Properties of $\concentration{\nbins}{\nballs}$}
In this section we consider the function $\concentration{\nbins}{\nballs}$ defined in \Cref{def:nu}. In \Cref{sub:well_definedness} we will show that $\concentration{\nbins}{\nballs}$ is well--defined and then in \Cref{sub:proof_nu} we will provide a proof of \Cref{lem:nu}.
\subsection{Well--definedness of $\concentration{\nbins}{\nballs}$}\label{sub:well_definedness}
We recall that for given $\nbins,\nballs\in \N$ we defined the function $f$ as
\begin{align}\label{eq:22}
f(x)=f_{\nbins,\nballs}(x):=x\log\nballs+x-\left(x+1/2\right)\log x-(x-1)\log \nbins. 
\end{align}
By basic calculus we obtain $f(x)>0$ for all $x\in (0,1]$, $f''(x)<0$ for all $x\geq 1$, and $f(x)\to -\infty$ as $x\to \infty$. This implies that $f$ has a unique zero in $(0, \infty)$, which shows that $\concentration{\nbins}{\nballs}$ is well--defined. Moreover, we obtain the following fact, which we will use in \Cref{sub:proof_nu}:
\begin{align}\label{eq:24}
f(x)
\begin{cases}
>0 &\text{if } x<\concentration{\nbins}{\nballs},\\
=0 &\text{if } x=\concentration{\nbins}{\nballs},\\
<0 &\text{if } x>\concentration{\nbins}{\nballs}.
\end{cases}
\end{align}
\subsection{Proof of \Cref{lem:nu}}\label{sub:proof_nu}
Throughout the proof, let $f$ as in (\ref{eq:22}). By definition, $\concentration{\nbins}{\nballs}$ is the unique positive zero of $f$ and for $x\in(0,1]$ we have $f(x)>0$, which implies \ref{lem:nu5}. 

In order to prove \ref{lem:nu7} we may assume that $\nballs\leq C\nbins^{1/3}$ for a suitable constant $C>0$. Now we get for $\nbins$ large enough
\begin{align*}
f(5/3)\lessorequal -1/9\log n+5/3\log C+5/3-13/6\log\left(5/3\right)<0.
\end{align*}
Together with (\ref{eq:24}) this implies $\concentration{\nbins}{\nballs}\leq 5/3$ for all $\nbins$ large enough, which yields \ref{lem:nu7}.

For \ref{lem:nu1} we may assume $\nballs=\Th{\nbins}$. Then we have for $a>0$
\begin{align*}
f\left(a\frac{\log n}{\log \log n}\right)=\left(1-a+\smallo{1}\right)\log n.
\end{align*}
Thus, \ref{lem:nu1} follows by (\ref{eq:24}). 

To prove \ref{lem:nu8}, we write $\nballs=c\nbins\log \nbins$ for $c=c(\nbins)=\smallo{1}$. We have for $a>0$ and $\nbins$ large enough
\begin{align*}
f(a\log \nbins)=\log \nbins\left(a\log c+a-a\log a+1\right)-1/2\log \log \nbins-1/2\log a<0,
\end{align*}
as $\log c\to -\infty$. Due to (\ref{eq:24}) this implies $\concentration{\nbins}{\nballs}<a\log \nbins$. As $a>0$ was arbitrary, we obtain $\concentration{\nbins}{\nballs}=\smallo{\log \nbins}$. 

For \ref{lem:nu6} we observe that by definition of $\Concentration=\concentration{\nbins}{\nballs}$
\begin{align*}
1=e\frac{\nballs}{\nbins\Concentration}\exp\left(\left(\log \nbins-1/2\log \Concentration\right)/\Concentration\right).
\end{align*}
Due to \ref{lem:nu8} we have $\left(\log \nbins-1/2\log \Concentration\right)/\Concentration=\smallomega{1}$, which yields $\Concentration=\smallomega{\nballs/\nbins}$.

For \ref{lem:nu4} we fix $\nbins\in \N$ and define $K(x):=\left(1+1/(2x)\right)\log x+\left(1-1/x\right)\log n-1$. It is easy to check that $K(\concentration{\nbins}{\nballs})=\log \nballs$ and $K$ is strictly increasing. This implies \ref{lem:nu4}.
 For \ref{lem:nu3} we let $\Concentration=\concentration{\nbins}{\nballs}$ and $\gamma\in\R$. Due to \ref{lem:nu1} we have  $\Concentration=\left(1+\smallo{1}\right)\log \nbins/\log\log \nbins$ and therefore we obtain 
\begin{align}\label{eq:1}
K\left(\Concentration+\gamma\right)-K(\Concentration)=\frac{\left(\log\log n\right)^2}{\log n}\left(\gamma +\smallo{1}\right).
\end{align}
On the other hand, we have 
\begin{align*}
K\big(\concentration{\nbins}{\nballs+d}\big)-K\big(\concentration{\nbins}{\nballs}\big)=\log (\nballs+d)-\log\nballs=\Th{d/\nballs}=\smallo{\left(\log \log \nbins\right)^2/\log \nbins}.
\end{align*}
Together with (\ref{eq:1}) this implies \ref{lem:nu3}, as $K$ is strictly increasing.

Similarly, we define for \ref{lem:nu9} the function $g(x):=\left(x+1/2\right)\log x-x$. Now \ref{lem:nu9} follows by the facts that $g(\specialconcentration{\nbins})=\log\nbins$ and $g$ is strictly increasing. 

Finally for \ref{lem:nu2}, let $\Concentration=\specialconcentration{\nbins}$ and $\gamma\in\R$. Then we have as $n \to \infty$
\begin{align}\label{eq:2}
g(\Concentration+\gamma)-g(\Concentration)=\left(\Concentration+1/2\right)\log \left(\frac{\Concentration+\gamma}{\Concentration}\right)+\gamma\log\left(\Concentration+\gamma\right)-\gamma \to
\begin{cases}
\infty &\text{if } \gamma>0\\
-\infty &\text{if } \gamma<0,
\end{cases}
\end{align} 
where we used $\Concentration=\smallomega{1}$ due to \ref{lem:nu1} and $\Concentration\log\left(\left(\Concentration+\gamma\right)/\Concentration\right)=\gamma+\smallo{1}$. We observe that $g(\specialconcentration{c\nbins})-g(\Concentration)=\log\left(c\nbins\right)-\log n=\Th{1}$. Together with (\ref{eq:2}) this implies $\specialconcentration{c\nbins}=\specialconcentration{n}+\smallo{1}$, as $g$ is strictly increasing.\qed
\end{document}